\documentclass[12pt]{article} 
\usepackage{amsfonts,amsmath,latexsym,amssymb,mathrsfs,amsthm,comment,setspace}
\usepackage{slashbox}
\usepackage{caption}

\let\OLDthebibliography\thebibliography
\renewcommand\thebibliography[1]{
  \OLDthebibliography{#1}
  \setlength{\parskip}{0pt}
  \setlength{\itemsep}{0pt plus 0.0ex}
}

\def\numberlikeadb{\global\def\theequation{\thesection.\arabic{equation}}}
\numberlikeadb
\newtheorem{theorem}{Theorem}[section]
\newtheorem{lemma}[theorem]{Lemma}
\newtheorem{corollary}[theorem]{Corollary}

\newtheorem{proposition}[theorem]{Proposition}
\newtheorem{remark}[theorem]{Remark}
\newtheorem{example}[theorem]{Example}

\usepackage{color} 
\newcommand{\rg}[1]{{\color{black} #1}}

\usepackage{lscape}
\usepackage{caption}
\usepackage{multirow}
\begin{document}

\title{Bounding Kolmogorov distances through Wasserstein and related integral probability metrics}
\author{Robert E. Gaunt\footnote{Department of Mathematics, The University of Manchester, Oxford Road, Manchester M13 9PL, UK, robert.gaunt@manchester.ac.uk; siqi.li-8@postgrad.manchester.ac.uk}\:\, and Siqi L$\mathrm{i}^{*}$}



\date{} 
\maketitle

\vspace{-5mm}


\begin{abstract}We establish general upper bounds on the Kolmogorov distance between two probability distributions in terms of the distance between these distributions as measured with respect to the Wasserstein or smooth Wasserstein metrics. These bounds \rg{generalise existing results from the literature.}
To illustrate the broad applicability of our general bounds, we apply them to extract Kolmogorov distance bounds from multivariate normal, beta and variance-gamma approximations that have been established in the Stein's method literature.
\end{abstract}

\noindent{{\bf{Keywords:}}}  Kolmogorov distance; Wasserstein distance; integral probability metric; inequality; approximation; Stein's method

\noindent{{{\bf{AMS 2010 Subject Classification:}}} Primary 60E15; 60F05; Secondary 41A10 


\section{Introduction}


Stein's method \cite{stein} is a powerful technique in probability theory for bounding the distance between two probability distributions with respect to a probability metric. It has found application throughout the mathematical sciences in areas as diverse as random graph theory \cite{bhj92}, queuing theory \cite{bd16} and analysis on Wiener space \cite{nourdin1}.  Stein's method is well established for normal and Poisson approximation (see the monographs \cite{bhj92,chen,np12}), and has been successfully applied to many other distributional limits; see 
\cite{ley,mrs21,ross}.


In much of the Stein's method literature, the focus is on bounding the distance between two probability distributions with respect to \emph{integral probability metrics} \cite{gs02,z93}: that is, for $\mathbb{R}^d$-valued random variables $X$ and $Y$,
\begin{equation*}d_{\mathcal{H}}(X,Y)=\sup_{h\in\mathcal{H}}|\mathbb{E}[h(X)]-\mathbb{E}[h(Y)]|
\end{equation*}
for some class of real-valued measurable test functions $\mathcal{H}\subset L^1(X)\cap L^1(Y)$. Here and throughout this paper, we use the standard abuse of notation $d_{\mathcal{H}}(X,Y)$ to denote \rg{$d_{\mathcal{H}}(\mu,\nu)$, where $\mu$ and $\nu$ denote the probability measures of the random variables $X$ and $Y$, respectively.}
 Common choices of $\mathcal{H}$ include
\begin{align*}\mathcal{H}_{\mathrm{K}}&=\{\mathbf{1}_{\cdot\leq z}\,|\,z\in\mathbb{R}^d\}, \\
\mathcal{H}_{\mathrm{W}}&=\{h:\mathbb{R}^d\rightarrow\mathbb{R}\,|\,\text{$h$ is Lipschitz, $\|h\|_{\mathrm{Lip}}\leq1$}\}, \\
\mathcal{H}_{\mathrm{bW}}&=\{h:\mathbb{R}^d\rightarrow\mathbb{R}\,|\,\text{$h$ is Lipschitz, $\|h\|\leq1$ and $\|h\|_{\mathrm{Lip}}\leq1$}\}, \\
\mathcal{H}_{[m]}&=\{h:\mathbb{R}\rightarrow\mathbb{R}\,|\,\text{$h^{(m-1)}$ is Lipschitz with $\|h^{(j)}\|\leq1$, $1\leq j\leq m$}\}, \\
\mathcal{H}_{m}&=\{h:\mathbb{R}\rightarrow\mathbb{R}\,|\,\text{$h^{(m-1)}$ is Lipschitz with $\|h^{(j)}\|\leq1$, $0\leq j\leq m$}\},
\end{align*}
where $h^{(0)}\equiv h$, $\|\cdot\|$ denotes the usual supremum norm of a real-valued function, and, for a Lipschitz function $h:\mathbb{R}^d\rightarrow\mathbb{R}$, we denote 
$\|h\|_{\mathrm{Lip}}=\sup_{x\not=y}\frac{|h(x)-h(y)|}{\|x-y\|_2}$ with $\|\cdot\|_2$ the Euclidean norm. Throughout the paper, $\|h^{(m)}\|$ will denote the Lipschitz constant of the $(m-1)$-th derivative of $h$. We note that for $h\in\mathcal{H}_m$ the function $h^{(m-1)}$ is differentiable almost everywhere. At certain points in this paper, the $m$-th derivative $h^{(m)}$ of $h\in\mathcal{H}_m$ will not exist at all points, but it will be well-defined in a piecewise sense on a finite number of intervals. The classes $\mathcal{H}_{\mathrm{K}}$, $\mathcal{H}_{\mathrm{W}}$ and $\mathcal{H}_{\mathrm{bW}}$ induce the Kolmogorov, Wasserstein (also known as the Kantorovich-Rubenstein or earth-mover's distance) and bounded Wasserstein distances (also known as the Fortet–Mourier or Dudley metric), which we denote by $d_{\mathrm{K}}$, $d_{\mathrm{W}}$ and $d_{\mathrm{bW}}$, respectively. The classes $\mathcal{H}_{[m]}$ and $\mathcal{H}_{m}$ induce smooth Wasserstein distances, which we denote by $d_{[m]}$ and $d_{m}$ respectively (see, for example, \cite{dp18} and \cite{amps16}). Note that $d_{[1]}=d_{\mathrm{W}}$ and $d_{1}=d_{\mathrm{bW}}$, and that $d_{m}\leq d_{[m]}$ for all $m\geq1$.
Other smooth Wasserstein distances can be induced by, for example, only requiring that $\|h^{(m)}\|\leq1$ (see, for example, \cite{f21}). A generalisation of the $d_{[m]}$ and $d_{m}$ metrics to $\mathbb{R}^d$-valued random vectors involves a little more notation, and is given in Section \ref{sec2.2}. 


One of the most basic applications of Stein's method is to derive error bounds for the classical central limit theorem. Let $X_1,\ldots,X_n$ be independent and identically distributed random variables with zero mean, unit variance and $\mathbb{E}[|X_1|^3]<\infty$. Let $W_n=\frac{1}{\sqrt{n}}\sum_{i=1}^nX_i$ and $Z\sim N(0,1)$. An application of Stein's method involving only elementary calculations yields the Wasserstein distance bound
\begin{equation}\label{dwclt}d_{\mathrm{W}}(W_n,Z)\leq\frac{1}{\sqrt{n}}\big(2+\mathbb{E}[|X_1|^3]\big).
\end{equation} 
(see \cite{r98} for a simple proof, and also \cite[Corollary 4.2]{chen} for the improved upper bound $\frac{1}{\sqrt{n}}\mathbb{E}[|X_1|^3]$). Optimal order Wasserstein distance bounds via Stein's method were first obtained by \cite{e74}, just a couple of years after Stein introduced his beautiful method. However, it took a further ten years until \cite{bh84} were able to use Stein's method to obtain optimal order $n^{-\frac{1}{2}}$ bounds for the central limit theorem in the Kolmogorov metric. Indeed, for technical reasons, it is often difficult to use Stein's method to directly obtain error bounds with respect to the Kolmogorov distance. It is, however, possible to deduce Kolmogorov distance bounds from Wasserstein distance bounds.

 Proposition 1.2, part 2, of \cite{ross} states that if $Y$ is a real-valued random variable with Lebesgue density bounded above by $C>0$, then for any real-valued random variable $X$,
\begin{equation}\label{rossbd}d_{\mathrm{K}}(X,Y)\leq\sqrt{2Cd_{\mathrm{W}}(X,Y)}.
\end{equation}
This is a useful bound because any bound obtained on the Wasserstein distance between an arbitrary distribution and one with bounded Lebesgue density, such as the normal or exponential distributions, immediately grants a bound in the Kolmogorov metric. Because of the significance of this implication, variants of the bound (\ref{rossbd}) in the setting of normal approximation are given in several other monographs and surveys on Stein's method; see \cite{bc05,c14,chen,np12,stein2}. Typically, inequality (\ref{rossbd}) will yield sub-optimal Kolmogorov distance bounds; for example, using the bound (\ref{dwclt}) yields a sub-optimal $O(n^{-\frac{1}{4}})$ Berry-Esseen bound for the central limit theorem. 

Over the years, many other bounds in the spirit of (\ref{rossbd}) have been derived, with some contributions including the following. Under the same assumptions as (\ref{rossbd}), \cite{pike} obtained the bound $d_{\mathrm{K}}(X,Y)\leq(\frac{C}{2}+1)\sqrt{d_{\mathrm{bW}}(X,Y)}$. For the case $Y$ has a variance-gamma distribution, \cite{gaunt vg2,gaunt vg3} obtained analogues of (\ref{rossbd}) in which the variance-gamma density has a singularity at the location parameter, whilst \cite{nnp16} have obtained analogues of (\ref{rossbd}) for the case that $Y$ is a mixture of normal distributions. 
Multivariate generalisations of (\ref{rossbd}) in which $Y$ has the multivariate normal distribution are given in \cite{app16,k19} with a similar bound on the convex distance recently obtained by \cite{npy20}. Other works in which smoothing techniques have been used to obtain bounds on $d_{\mathcal{H}}(X,Y)$, where $Y$ has the multivariate normal distribution and $\mathcal{H}$ is a class of non-smooth test functions include \cite{b86,gotze,Reinert Multivariate,rr96}. For the Dirichlet distribution see \cite{grr17}.  The case that $Y$ is a Gaussian process has also recently been considered by \cite{brz21}, whilst the Dirichlet process is dealt with in \cite{gr21}. Upper bounds on the Wasserstein distance $d_\mathrm{W}=d_{[1]}$ in terms of the $d_{[2]}$ and $d_{[3]}$ metrics have also been given in \cite{dp18,h20,npr10}, and a bound on the total variation distance in terms of the Wasserstein metric when $X$ and $Y$ are random variables belonging to a finite sum of Wiener chaoses is given by \cite{npoly13}. 


In the light of these
results and the significance of the bound (\ref{rossbd}), in this paper, we address the natural problem of establishing general bounds that allow one to extract Kolmogorov distance bounds from Wasserstein and smooth Wasserstein distance bounds. Through our general bounds, we aim to reduce the need for researchers to on a case-by-case basis establish bounds that enable their smooth Wasserstein distance bounds to be converted into Kolmogorov distance bounds.  However, 
in certain applications,
one may wish to make use of particular structures of the problem at hand to potentially obtain sharper bounds than would result from our general bounds, so this goal is unlikely to be fully met.
Nevertheless, our general bounds should at least serve a useful purpose of allowing researchers to immediately extract Kolmogorov distance bounds from bounds they obtain with respect to a smooth Wasserstein metric. It should also be noted that whilst we have provided our motivation through Stein's method, Wasserstein and smooth Wasserstein metrics arise throughout probability and statistics,
meaning that our bounds may find utility in other research domains.

In Section \ref{sec2}, we provide general upper bounds for the Kolmogorov distance $d_{\mathrm{K}}(X,Y)$ in terms of $d_{m}(X,Y)$, $m\geq1$. Since $d_{m}\leq d_{[m]}$, upper bounds in terms of $d_{[m]}(X,Y)$ follow immediately. Propositions \ref{prop2.1} and \ref{prop2.2} cover the case that $X$ and $Y$ are real-valued random variables, whilst bounds for $\mathbb{R}^d$-valued random vectors $X$ and $Y$ are given in Proposition \ref{prop2.4}.  Our bounds hold for an arbitrary random element $X$, whilst we assume that the Lebesgue density of $Y$ is either bounded or has a certain type of behaviour at its singularities (logaritmic, power function, or product of logarithms and powers functions); this covers a wide class of distributions, see Remarks \ref{remappl1} and \ref{remappl2}. A bound for the important case of the multivariate normal distribution is given in Proposition \ref{prop2.5}. \rg{It is therefore apparent that our bounds are rather general; for example, we have extended the scope of inequality (\ref{rossbd}) to generalise the Wasserstein distance $d_{\mathrm{W}}$ to the $d_m$ metric for $m\geq1$; to allow for random variables $Y$ with unbounded densities; and to multivariate distributions.} In Section \ref{sec3}, we illustrate how the general bounds of Section \ref{sec2} can be applied to specific distributional approximations by deducing Kolmogorov distance bounds from bounds in the Stein's method literature for multivariate normal, beta and variance-gamma approximation that were given with respect to the $d_m$ metric for $m\geq1$. The proofs of the results from Section \ref{sec2} involve the construction of suitable smooth approximations to indicator functions; the details are worked out in Section \ref{sec4}. In Section \ref{sec5}, we use these results to prove the general bounds of Section \ref{sec2}.







\section{General bounds}\label{sec2}

In this section, we state general results for bounding the Kolmogorov distance between the distributions of the random elements $X$ and $Y$ in terms of the distance between these elements as measured by the $d_m$ metrics, $m \geq 1.$ 
In each of our bounds, $X$ is any random variable/vector (only needing to be such that $d_m(X,Y)$ is well-defined), whilst various assumptions are made on the density of $Y$. The proofs are given in Section 5.

\subsection{The univariate case}\label{sec2.1}	

General bounds for univariate distributions are given in the following proposition. Slightly larger and less compact bounds that have a larger range of validity (that is weaker assumptions on $d_m(X,Y)$) are given in Proposition \ref{prop2.2} below.

\begin{proposition}\label{prop2.1} Let $X$ be any real-valued random variable and let $Y$ be a continuous real-valued random variable with probability density function $p$. For $m\geq1$, let $M_m=2^{m-2}(m-1)!$ and $N_m = 2^m M_m.$

\vspace{2mm}

\noindent{(i):} Suppose there exists a positive constant $A>0$ such that $p(y) \leq A $ for all $y \in \mathbb{R}.$ Suppose also that $d_m(X,Y) \leq \frac{A}{2N_m}. $ Then
\begin{equation}\label{dmfirst}
d_{\mathrm{K}}(X,Y) \leq 2\big(A^m M_m  d_{m} (X,Y)\big)^{\frac{1}{m+1}}.
\end{equation}

Suppose now that the density function $p(y)$ of the random variable $Y$ has singularities at the points $y_1,\ldots,y_n \in \mathbb{R}.$ Let $\epsilon>0$ be a constant, which satisfies $\epsilon\leq\frac{1}{2}\min_{1\leq i<j\leq n}|y_i-y_j|$ if $n\geq2$. 



\vspace{2mm}

\noindent{(ii):} 
Suppose that there exist constants $A,c>0$ such that $p(y)\leq -A\log|c(y-y_i)|$ for all $|y-y_i| < \epsilon$, $i=1,\ldots,n$, where $\epsilon\leq\frac{1}{c}$. Suppose also that  $\int_I p(y) \,dy \leq \int^{\delta}_{-\delta} -A \log |cy|\,dy$ for any interval $I$ with length $2\delta\leq 2\epsilon$. 
Then, if $d_m(X,Y) \leq \frac{A}{N_{m}} \min(1,(2\epsilon)^{m+1})$, 
\begin{equation}\label{dmsecond}
d_{\mathrm{K}}(X,Y) \leq  \bigg[2+\frac{1}{m+1} \log \bigg(\frac{2A}{c^{m+1}M_md_m(X,Y)}\bigg)\bigg]\big(A^mN_md_m(X,Y)\big)^{\frac{1}{m+1}}.
\end{equation}


\noindent{(iii):} Suppose there exist constants $A>0$ and $ 0<a<1$ such that $p(y)\leq A|y-y_i|^{-a}$ for all $|y-y_i| < \epsilon$, $i=1,\ldots,n$. Suppose also that $\int_I p(y) \,dy \leq \int^{\delta}_{-\delta} A|y|^{-a}\,dy  $ for any interval $I$ with length $2\delta\leq2\epsilon$. Then, if $d_m(X,Y) \leq \frac{2^aA}{(1-a) N_{m} } \min(1,(2\epsilon)^{m+1-a}) $,
\begin{equation}\label{dmthird}
d_{\mathrm{K}}(X,Y) \leq 2\bigg(\frac{2^aA}{1-a}\bigg)^{\frac{m}{m+1-a}}\big(N_md_m(X,Y)\big)^{\frac{1-a}{m+1-a}}.
\end{equation}

\noindent{(iv):} Suppose that there exist constants $A,c>0$, $0\leq a<1$ and $b\geq0$ such that $p(y)\leq A |y-y_i|^{-a}(-\log |c(y-y_i)|)^b $ for all $|y-y_i| < \epsilon$, $i=1,\ldots,n$, where $\epsilon\leq\frac{1}{c}$. Suppose also that $\int_I p(y)\,dy \leq \int^{\delta}_{-\delta} A|y|^{-a}(-\log |cy|)^b  \,dy $ for any interval $I$ with length $2\delta\leq2\epsilon$. Then, if $d_m(X,Y)< \frac{2^{a+b+1}c^{1-a}A}{(1-a) N_{m} }\min(1,(2\epsilon)^{m+1-a},2^{-(a+b)(1+\frac{m}{1-a})}(1/c)^{m+1-a} )$,
\begin{align}
 d_{\mathrm{K}}(X,Y)  &\leq \bigg(\frac{2^{a+b+1}A}{1-a}\bigg)^{\frac{m}{m+1-a}}(N_md_m(X,Y))^{\frac{1-a}{m+1-a}}\nonumber\\
 \label{dmfourth}&\quad\times\bigg[1+\frac{1}{{(m+1-a)^b}}\log^b \bigg(\frac{2^{b+2}A}{(1-a)c^{m+1-a}M_md_m(X,Y)}\bigg) \bigg] .
\end{align}




\end{proposition}


\begin{remark}\label{remappl1} (1): In the Stein's method literature, bounds are often given for the quantity $|\mathbb{E}[h(X)]-\mathbb{E}[h(Y)]|$, where $h$ is a real-valued function. For $h$ such that $h^{(m-1)}$ is Lipschitz, bounds are often of the form $|\mathbb{E}[h(X)]-\mathbb{E}[h(Y)]| \leq \sum^m_{k=0} a_k \|h^{(k)} \|$, for $a_0,\ldots,a_m\geq0$. 
Restricting $h$ to the class $\mathcal{H}_m$, we get that $d_m(X,Y)\leq\sum_{k=0}^ma_k$, and so we can apply Proposition \ref{prop2.1} to immediately deduce a bound on $d_{\mathrm{K}}(X,Y)$.  Doing so will, however, typically lead to worse bounds (in the constant) than applying the smoothing methods of this paper directly to the quantity $\sum^m_{k=0} a_k \|h^{(k)} \|$.

\vspace{2mm}

\noindent{(2):} Since $d_1\leq d_{[1]}=d_{\mathrm{W}}$ and $d_m\leq d_{[m]}$ for all $m\geq1$, inequality (\ref{dmfirst}) generalises inequality (\ref{rossbd}) to the $d_{[m]}$ metric for $m\geq1$. 
We expect that inequality (\ref{dmfirst}) will typically give sub-optimal Kolmogorov distance rates, but that the exponent $\frac{1}{m+1}$ is optimal. We have not been able to provide an example to confirm this, although in Remark \ref{remnearopt} 
 we show that the exponent $q$ in the general bound $d_{\mathrm{K}}(X,Y)\leq C(d_m(X,Y))^q$ must satisfy $q\leq\frac{1}{m}$. \rg{We further remark that for $m\geq2$ the upper bound in inequality (\ref{dmfirst}) has a slower convergence rate than that of inequality (\ref{rossbd}). However, an advantage of inequality (\ref{dmfirst}) is that if one only has access to a bound on the distance between the distributions of $X$ and $Y$ in the weaker $d_m$ metric for some $m\geq2$, but not in the Wasserstein metric, then one can make use of inequality (\ref{dmfirst}), whilst inequality (\ref{rossbd}) cannot be applied.}

\vspace{2mm}

\noindent{(3):} Let us make a remark about part (iii) of Proposition \ref{prop2.1}; a similar comment applies to parts (ii) and (iv). If there is just one singularity ($n=1$) and it is assumed the density $p(y)$ is non-decreasing on $(-\infty,y_1)$ and non-increasing on $(y_1,\infty)$, then the condition that $\int_I p(y) \,dy \leq \int^{\delta}_{-\delta} A|y|^{-a}\,dy  $ for any interval $I$ with length $2\delta\leq2\epsilon$  is automatically satisfied, provided the assumption that there exist constants $A>0$ and $ 0<a<1$ such that $p(y)\leq A|y-y_1|^{-a}$ for all $|y-y_1| < \epsilon$ is met. 

\vspace{2mm}

\noindent{(4):} The assumptions on $d_m(X,Y)$ in parts (i)--(iii) are quite mild and if the constant $A$ is sufficiently large then the condition becomes trivial. For example, taking $d_m(X,Y)=\frac{A}{2N_m}$ in inequality (\ref{dmfirst}) gives the bound $d_{\mathrm{K}}(X,Y)\leq A$, which is uninformative if $A\geq1$. The assumptions on $d_m(X,Y)$ in part (iv) are more restrictive, which results from an application of an upper bound on the upper incomplete function that only holds for certain parameter values. As for parts (i)--(iii), if $A$ is sufficiently large then the condition for $d_m(X,Y)$ becomes trivial, though. Also, in a typical application of Stein's method, one is interested in the situation that $d_m(X,Y)$ is `small', for example in deriving bounds on rates of convergence, in which case making such assumptions on $d_m(X,Y)$ are not restrictive.  



\vspace{2mm}

\noindent{(5):} The bounds of Proposition \ref{prop2.1} can be applied in many settings. Part (i) applies when $Y$ has bounded Lebesgue density, a condition satisfied by many classical distributions, such as the normal and exponential. Distributions with logarithmic singularities (part (ii)) include the product of two correlated zero-mean normal random variables (see \cite{gaunt 21}) and a subclass of the variance-gamma distribution (see Section 3). Distributions with power law singularities (part (iii)) include subclasses of the gamma and beta distributions (see Section 3). Distributions with more general singularities of the form given in part (iv) include the product of $k \geq 2$ independent standard normal random variables (see \cite{st70}) and the product of two independent gamma random variables with equal shape parameters (see \cite{m68}). 

\vspace{2mm}

\noindent{(6):} The bound in part (iv) is of  order  $O((d_m(X,Y))^{\frac{1-a}{m+1-a}}\log^b (\frac{1}{d_m(X,Y)}) )$ for small $d_m(X,Y).$ Setting $a=0$ and or $b=0$ then yields  bounds that are of the same order as those in parts (i), (ii) and (iii), although the bound does not reduce exactly to those bounds due to extra approximations applied as part of the derivation. Observe that the rate of convergence of the bound decreases as $m,a$ and $b$ increase.
\end{remark}

\begin{proposition}\label{prop2.2} Let $X$ be any real-valued random variable and let $Y$ be a real-valued random variable with probability density function $p$. Let $m\geq1$.

\vspace{2mm}

\noindent{(i):} Suppose that the density $p$ is bounded above by the constant $A>0$. 
Then
\begin{equation*}\label{dmfirst00}
 d_{\mathrm{K}}(X,Y) \leq 2\big(A^m M_m  d_{m} (X,Y)\big)^{\frac{1}{m+1}}+M_md_m(X,Y).
\end{equation*}





\noindent{(ii):} 
Suppose that there exist constants $A,c>0$ and $B\geq0$ such that  $\int_I p(y) \,dy \leq \int^{\delta}_{-\delta} (-A \log_{-} |cy|+B)\,dy$ for any interval $I$ with length $2\delta>0$, where $\log_{-}(x)=\min(0,\log (x))$. 
Then 
\begin{align*}\label{dmsecond00}
d_{\mathrm{K}}(X,Y) &\leq  \bigg[2+\frac{B}{2A}+\frac{1}{m+1} \log_{-} \bigg(\frac{2A}{c^{m+1}M_md_m(X,Y)}\bigg)\bigg]\big(A^mN_md_m(X,Y)\big)^{\frac{1}{m+1}} \\
&\quad+M_md_m(X,Y).
\end{align*}


\noindent{(iii):} Suppose there exist constants $A>0$ and $ 0<a<1$ such that $\int_I p(y) \,dy \leq \int^{\delta}_{-\delta} A|y|^{-a}\,dy  $ for any interval $I$ with length $2\delta>0$. Then
\begin{equation*}\label{dmthird00}
d_{\mathrm{K}}(X,Y) \leq 2\bigg(\frac{2^aA}{1-a}\bigg)^{\frac{m}{m+1-a}}\big(N_md_m(X,Y)\big)^{\frac{1-a}{m+1-a}} +M_md_m(X,Y).
\end{equation*}

\noindent{(iv):} Suppose that there exist constants $A,c>0$, $B\geq0$, $0\leq a<1$ and $b\geq0$ such that $\int_I p(y)\,dy \leq \int^{\delta}_{-\delta}[ A|y|^{-a}(-\log_{-} |cy|)^b+B]  \,dy $ for any interval $I$ with length $2\delta>0$. Then, if $d_m(X,Y)< \frac{c^{-m}A}{(1-a) M_{m} }\cdot2^{1-\frac{m(b+1)}{1-a}}$,
\begin{align*}
 d_{\mathrm{K}}(X,Y)  &\leq \bigg(\frac{2^{a+b+1}A}{1-a}\bigg)^{\frac{m}{m+1-a}}(N_md_m(X,Y))^{\frac{1-a}{m+1-a}}\times\nonumber\\
 \label{dmfourth00}&\quad\times\bigg[1+\frac{1}{{(m+1-a)^b}}\log_{-}^b \bigg(\frac{2^{b+2}A}{(1-a)c^{m+1-a}M_md_m(X,Y)}\bigg) \bigg]  \\
 &\quad+\frac{B}{2}\bigg(\frac{(1-a)N_md_m(X,Y)}{2^{a+b+1}A}\bigg)^{\frac{1}{m+1-a}}+M_md_m(X,Y).
\end{align*}
\end{proposition}

\subsection{The multivariate case}\label{sec2.2}
As part of our proofs of Propositions \ref{prop2.1} and \ref{prop2.2}, for fixed $\alpha > 0$, we required a bound on the probability $\mathbb{P}(z \leq Y \leq z+\alpha )$ for all $z \in \mathbb{R}.$ Here and elsewhere in the paper, we abuse notation and let $z+\alpha$ denote the $d$-dimensional vector with $j$-th component $z_j+\alpha$. For general random variables with densities satisfying the assumptions of parts (i)--(iv), it was possible to get accurate bounds on this probability. However, obtaining bounds on a suitable generalisation of this probability for general $\mathbb{R}^d$-valued random vectors $Y$ with a good dependence on the dimension $d$ is difficult. As such, we bound the probability in terms of $\alpha$ and a constant that may depend on $d$, and so the bounds in our multivariate generalisation of Propositions \ref{prop2.1} and \ref{prop2.2} (given below) do not have explicit constants. For certain specific distributions, such as the multivariate normal distribution, it may be possible to obtain accurate bounds with a good dependence on $d$, in which case explicit constants can be given; see Proposition \ref{prop2.5} below. 


Let $x^{(j)}$ denote the $j$-th component of the vector $x$, and denote the open box centered at $z\in\mathbb{R}^d$ with `width' $2r>0$ by $B(z,r) = \{ x \in \mathbb{R}^d: \max_{1 \leq j \leq d}|z^{(j)}-x^{(j)}| <r \}$.

\begin{proposition}\label{prop2.4}
Let $X$ be any $\mathbb{R}^d$-valued random vector and let $Y$ be a continuous $\mathbb{R}^d$-valued random vector with probability density function $p$. Let $m\geq1$.

\vspace{2mm}

\noindent{(i):} Suppose $p$ is bounded. Then, there exists a universal constant $C>0$ such that
\begin{equation*}
d_{\mathrm{K}}(X,Y) \leq C \big(d_m(X,Y)\big)^{\frac{1}{m+1}}.
\end{equation*}

\noindent{(ii):} 
Suppose there exist constants $A_1,\ldots,A_d,B\geq0$ and $c>0$  such that $\int_{B(z,\delta)} p(y) \,dy  \leq \int_{B(0,\delta)}  \sum^d_{j=1}(-A_j\log_{-}|cy^{(j)}| +B) \,dy$ for any open box $B(z,\delta)\subset\mathbb{R}^d$. Then, there exist universal constants $C_1,C_2>0$ such that
\begin{equation}\label{dmfourth2}
d_{\mathrm{K}}(X,Y) \leq C_1 \big(d_m(X,Y)\big)^{\frac{1}{m+1}}\log\bigg(\frac{C_2}{d_m(X,Y)}\bigg).
\end{equation}

\noindent{(iii):} 
Suppose there exist constants $A_1,\ldots,A_d,B\geq0$ and $0 < a < 1 $  such that $\int_{B(z,\delta)} p(y) \,dy  \leq \int_{B(0,\delta)}  \sum^d_{j=1}A_j |y^{(j)}|^{-a} \,dy$ for any open box $B(z,\delta)\subset\mathbb{R}^d$. Then, there exists a universal constant $C>0$ such that
\begin{equation}\label{dmfourth2dd}
d_{\mathrm{K}}(X,Y) \leq C \big(d_m(X,Y)\big)^{\frac{1-a}{m+1-a}}.
\end{equation}
\end{proposition}

\begin{remark}\label{remappl2} 
Like Propositions \ref{prop2.1} and \ref{prop2.2}, the bounds of Proposition \ref{prop2.4} are widely applicable. Part (i) of Proposition \ref{prop2.4} is applicable if the density of the random vector $Y$ is bounded. Part (ii) is applicable to a subclass of the multivariate variance-gamma distribution (also known as the generalized asymmetric multivariate Laplace distribution \cite[Chapter 6]{kkp01}). The bound of part (iii) also applies to another subclass of the multivariate variance gamma distribution, as well as the Dirichlet and inverted Dirichlet distributions.
\end{remark}

 When $Y$ is a multivariate normal random vector, we can bound $d_{\mathrm{K}}(X,Y)$ in terms of $d_m(X,Y)$ with an explicit constant. We are able to do so due to the following inequality.  
 
 \vspace{2mm}

\noindent{\emph{Nazarov's inequality \cite{n03}; Lemma A.1 in \cite{cck17} with detailed proof in \cite{cck17b}}: Let $Y =(Y_1,\ldots,Y_d)^T $ be a centered multivariate normal random vector in $\mathbb{R}^d$ such that $\mathbb{E}[Y_j^2] \geq \sigma^2$ for all $1 \leq j \leq d$ and some $\sigma > 0.$ Then, for every $z \in \mathbb{R}^d$ and $\alpha>0$,
\begin{equation}\label{eq:7}
\mathbb{P}(Y \leq z + \alpha) - \mathbb{P}(Y \leq z) \leq \frac{\alpha}{\sigma}(\sqrt{2 \log d}+2).
\end{equation}
The upper bound 
is of the optimal order with respect to the dimension $d$ (see \cite{b93}). 

Before stating the proposition, we introduce some notation. 
For a function $h:\mathbb{R}^d\rightarrow\mathbb{R}$, we abbreviate  $|h|_0:=\|h\|$ and $|h|_k:=\max_{ i_1,\ldots,i_k}\big\|\frac{\partial^k h(w)}{\prod^k_{j=1} \partial w_{i_j}}\big\|$, $k\geq1$. If $h:\mathbb{R}^d\rightarrow\mathbb{R}$ is such that $h^{(m-1)}$ is Lipschitz, 
then we understand $|h|_m$ to represent the largest Lipschitz constant of the functions $\frac{\partial^{m-1} h(w)}{\prod^{m-1}_{j=1} \partial w_{i_j}}$, $1\leq i_1,\ldots,i_{m-1}\leq d$. With this notation we introduce the following multivariate generalisations of the functions classes $\mathcal{H}_{[m]}$ and $\mathcal{H}_m$:
\begin{align*}\mathcal{H}_{[m]}&=\{h:\mathbb{R}^d\rightarrow\mathbb{R}\,:\,\text{$h^{(m-1)}$ is Lipschitz with $|h|_j\leq1$, $1\leq j\leq m$}\}, \\
\mathcal{H}_{m}&=\{h:\mathbb{R}^d\rightarrow\mathbb{R}\,:\,\text{$h^{(m-1)}$ is Lipschitz with $|h|_j\leq1$, $0\leq j\leq m$}\}
\end{align*} 
(note that we suppress the dependence on the dimension $d$ in our notation for these functions classes). We denote the integral probability metrics induced by these function classes by $d_{[m]}$ and $d_m$, respectively. As in the univariate case, we note that $d_m\leq d_{[m]}$ for all $m\geq1$ and that $d_{1}=d_{\mathrm{bW}}$ and $d_{[1]}=d_{\mathrm{W}}$. We also let $M_m'$ be given by $M_m'=2^{m-2}(m-1)!$ for $1\leq m\leq 4$, and $M_m'=2^{\frac{3m}{2}-2}(m-1)!$ for $m\geq5$, and let $N_m'=2^m M_m'$. 

\begin{proposition}\label{prop2.5}
Let $X$ be any random vector in $\mathbb{R}^d,$ and let $m\geq1$.  Let $Y\sim \mathrm{MVN}_d(\mu,\Sigma)$ be a multivariate normal random vector with mean vector $\mu \in \mathbb{R}^d$ and covariance matrix $\Sigma$. Suppose that $\sigma_{jj} = (\Sigma)_{jj} >0 $ for all $1 \leq j \leq d.$ Let $\sigma = \min_{1\leq j \leq d} \sigma_{jj}. $ Then, with the second bound holding if $d_m(X,Y) \leq \frac{2+\sqrt{2\log d}}{\sigma N_{m}'}$ (if $m\geq2$ this can be weakened to $d_m(X,Y) \leq \frac{4+2\sqrt{2\log d}}{\sigma M_{m}'}$),
\begin{align}\label{multin0}
d_{\mathrm{K}}(X,Y) &\leq 2 \bigg(\frac{\sqrt{2 \log d}+2}{\sigma}\bigg)^{\frac{m}{m+1}} \big(N_{m}' d_m(X,Y)\big)^{\frac{1}{m+1}}+M_m'd_m(X,Y), \\
\label{multin}
d_{\mathrm{K}}(X,Y)& \leq 2\bigg(\frac{\sqrt{2 \log d}+2}{\sigma}\bigg)^{\frac{m}{m+1}} \big(N_{m}' d_m(X,Y)\big)^{\frac{1}{m+1}}.
\end{align}
\end{proposition}

\begin{remark}(1): The assumption $d_m(X,Y) \leq \frac{2+\sqrt{2\log d}}{\sigma M_{m}'}$ for inequality (\ref{multin}) is mild.  Indeed, applying inequality (\ref{multin}) with $d_m(X,Y) = \frac{2+\sqrt{2\log d}}{\sigma M_{m}'}$ yields the bound $d_\mathrm{K}(X,Y)\leq \frac{2(\sqrt{2\log d}+2)}{\sigma}$, which is a trivial bound if $\sigma\leq 2(\sqrt{2\log d}+2)$.

\vspace{2mm}


\noindent{(2):} Since $d_{1}\leq d_{[1]}$, inequality (\ref{multin}) generalises a bound of \cite{k19} (which improved on a bound of \cite{app16}) that bounds $d_{\mathrm{K}}(X,Y)$ in terms of $d_{\mathrm{W}}(X,Y)=d_{[1]}(X,Y)$ when $Y$ is multivariate normal random vector. 
\end{remark}

\section{Examples}\label{sec3}

\rg{In this section, we provide several applications of the general bounds of Section \ref{sec2}.
These examples are chosen because of their inherent interest and also to 
serve as useful illustrations in the application of the general bounds in particular settings. Indeed, our examples have been chosen to demonstrate the broad applicability of our general results; we consider a multivariate example (multivariate normal approximation in Section \ref{sec3.2}), a distribution with multiple singularities (beta approximation in Section \ref{sec3.3}), and a distribution for which the density is either bounded or has a logarithmic or power law singularity depending on the parameter values (variance-gamma approximation in Section \ref{sec3.5}). Our general bounds are by no means restricted to these distributional approximations, and this section serves as an example of how efficiently one can use our bounds to extract Kolmogorov distance bounds for distributional approximations in which this would otherwise by technically demanding to do so, with the cost of a typically sub-optimal rate. In Section \ref{sec3.1}, we also provide an application that demonstrates that our bounds have utility beyond the derivation of Kolmogorov distance bounds, by enabling us to provide an efficient proof of a result concerning the regularity of the solution of the standard normal Stein equation, a fundamental object in Stein's method.
}


\subsection{The solution of the standard normal Stein equation}\label{sec3.1}

The example in this section differs from those of Sections \ref{sec3.2}--\ref{sec3.5} in that we do not apply the bounds of Section \ref{sec2} to bound Kolmogorov distance between two probability distributions. Instead, we use Proposition \ref{prop2.1} to prove a \rg{result, that to the best of our knowledge has not previously been stated in the literature,}
for the solution of the standard normal Stein equation, an object that lies at the heart of Stein's method for normal approximation. To motivate our result, we recall that, for a suitable real-valued test function $h:\mathbb{R} \rightarrow \mathbb{R}, $ the function
\begin{equation}\label{steinsoln}
f_h(x) = e^{\frac{x^2}{2}} \int^x_{-\infty}e^{-\frac{t^2}{2}} \{h(t) - \mathbb{E}[h(Z)] \}\, dt
\end{equation}
solves the standard normal Stein equation $f'(x) -xf(x) = h(x) - \mathbb{E}[h(Z)] $, where $Z \sim N(0,1)$ (see \cite{stein,stein2}). Let $k \geq 0$, and write $h^{(0)} \equiv 0. $ If $h^{(k)}$ is Lipschitz, then the following bounds hold (see \cite{gr96, gaunt gaunt rates,daly08}, respectively): for $k \geq 0$,
\begin{equation*}
 \| f_{\rg{h}}^{(k)} \| \leq \frac{\| h^{(k+1)} \|}{k+1}, \quad \| f_{\rg{h}}^{(k+1)} \| \leq \frac{\Gamma(\frac{k+1}{2})  \| h^{(k+1)} \|}{\sqrt{2} 
  \Gamma(\frac{k}{2}+1) },
 \quad \|f_{\rg{h}}^{(k+2)} \| \leq 2 \|h^{(k+1)} \|. 
\end{equation*}	
It is natural to ask whether bounds of the form $\|f^{(k+r)} \| \leq C_{k,r} \| h^{(k+1)} \| $ hold for $r \geq 3.$ The following proposition asserts that this is not possible for $r \geq 4.$ We expect that this is also not possible for $r\geq3$, but to prove this a more refined analysis would be required. It is, however, interesting to note how efficiently we are able to prove Proposition \ref{prop2.8} using Proposition \ref{prop2.1}, and we also note that similar arguments could in principle be applied to prove analogues of Proposition \ref{prop2.8} for solutions of other Stein equations.

\begin{proposition}
Let $f_h$ denote the solution (\ref{steinsoln}) of the standard normal Stein equation. Let $k \geq 0$ and $r \geq 4.$ Then there does not exist a constant $C_{k,r} > 0 $ such that $\|f_{\rg{h}}^{(k+r)} \| \leq  C_{k,r} \|h^{(k+1)} \| $ for all $h: \mathbb{R} \rightarrow \mathbb{R} $ such that $h^{(k)}$ is Lipschitz.
\end{proposition}\label{prop2.8}

\rg{\begin{remark} We expect that a stronger result holds in that, under the assumptions of Proposition \ref{prop2.8}, the function $f_h^{(k+r-1)}$ is not Lipschitz for $k\geq0$ and $r\geq4$. Moreover, in light of the discussion proceeding Proposition \ref{prop2.8}, we except that $f_h^{(k+r-1)}$ is not Lipschitz for $k\geq0$ and $r\geq3$.
\end{remark}
}

We will need the following lemma, which can be read off as an intermediate bound in the proof of Theorem 3.1 of \cite{gaunt gaunt rates}.  

\begin{lemma}[Gaunt \cite{gaunt gaunt rates}]
Let $X_1,\ldots,X_n$ be independent and identically distributed random variables such that $\mathbb{E}[|X_1|^{p+1}]<\infty,$ $p \geq 2.$ Suppose that $\mathbb{E}[X_1^k] =\mathbb{E}[Z^k] $ for $1 \leq k \leq p.$ Define $W_n = \frac{1}{\sqrt{n}} \sum^n_{i=1} X_i.$ Let $f_h$ be the solution (\ref{steinsoln}) of the standard normal Stein equation. Then 
\begin{equation}\label{iid example}
|\mathbb{E}[h(W_n)]-\mathbb{E}[h(Z)]| \leq \frac{\|f_h^{(p)} \|}{n^{\frac{p-1}{2}}}\bigg\{\frac{\mathbb{E}[|X_1|^{p-1}]}{(p-1)!}+\frac{\mathbb{E}[|X_1|^{p+1}]}{p!}\bigg\}.
\end{equation}
\end{lemma}
 \noindent{\emph{Proof of Proposition \ref{prop2.8}.}} 
Let $h \in \mathcal{H}_{p-t}$ for $t \geq 0.$ Suppose there exists a constant $C_p>0$, such that $\|f_h^{(p)} \| \leq C_p\| h^{(p-t)} \|. $ As $h \in \mathcal{H}_{p-t},$ $\| f_h^{(p)} \| \leq C_p$ and inserting this inequality into (\ref{iid example}) gives the bound
\begin{equation*}
d_{p-t} (W_n,Z) \leq \frac{C_p}{n^{\frac{p-1}{2}}}\bigg\{\frac{\mathbb{E}[|X_1|^{p-1}]}{(p-1)!}+\frac{\mathbb{E}[|X_1|^{p+1}]}{p!}\bigg\}.
\end{equation*} 
Therefore, by the bound of part (i) of Proposition \ref{prop2.1}, we obtain a bound for $d_{\mathrm{K}}(W_n,Z)$ that is order $n^{-\frac{p-1}{2(p-t+1)}}.$ To be consistent with the optimal Berry-Esseen rate of convergence of order $n^{-\frac{1}{2}}$ we require that $t \leq 2$; that is if $t \geq 3$ we have a contradiction.
\hfill $\Box$

\begin{remark}\label{remnearopt} The argument used to prove Proposition \ref{prop2.8} can be 
reversed to prove that, under the assumptions of part (i) of Proposition \ref{prop2.1}, the exponent $q$ in the inequality $d_{\mathrm{K}}(X,Y)\leq C(d_m(X,Y))^q$ must satisfy $q\leq\frac{1}{m}$, for all $m\geq1$. Indeed, applying this inequality to (\ref{iid example}) with $p=m+1\geq2$ and the bound $\|f_{\rg{h}}^{(m+1)}\|\leq 2\|h^{(m)}\|$ of \cite{daly08} gives that $d_{\mathrm{K}}(W_n,Z)\leq C(d_m(W_n,Z))^q\leq C'n^{-\frac{mq}{2}}$, and to be consistent with the $n^{-\frac{1}{2}}$ Berry-Esseen rate we require that $q\leq\frac{1}{m}$.
\end{remark}

\subsection{Multivariate normal approximation}\label{sec3.2}
In the Stein's method literature, there are many examples of bounds given with respect to the $d_m$ metrics for multivariate normal approximation to which Proposition \ref{prop2.5} could be applied to deduce Kolmogorov distance bounds. Out of this extensive literature, we elect to give an application to a widely-used general bound of \cite{Reinert Multivariate}.

\begin{theorem}[Reinert and R\"ollin \cite{Reinert Multivariate}]\label{prop2.7} Suppose that $(W,W')$ is an exchangeable pair of $\mathbb{R}^d$-valued random vectors such that $ \mathbb{E}[W] = 0$ and $ \mathbb{E}[WW^T] = \Sigma$,
with $\Sigma \in \mathbb{R}^{d \times d}$ symmetric and positive definite. Suppose further that $\mathbb{E}^W[W'-W]  = - \Lambda W+R $, where $\Lambda$ is an invertible $d \times d$ matrix and $R$ is a $\sigma(W)-$measurable remainder random vector. Let $Y\sim\mathrm{MVN}_d(0,\Sigma)$. Then, for every three times differentiable function $h$,  
\begin{equation*}
|\mathbb{E}[h(W)]-\mathbb{E}[h(Y)]| \leq \frac{|h|_2}{4}A+\frac{|h|_3}{12}B+\bigg(|h|_1 + \frac{1}{2}d \|\Sigma \|^{\frac{1}{2}} |h|_2 \bigg)C,
\end{equation*} 
where $\|\Sigma\|$ is the supremum norm of the matrix $\Sigma$, and also $\lambda^{(i)}= \sum^d_{m=1}|(\Lambda^{-1})_{m,i}|,$ 
\begin{equation*}
\begin{split}
 A &= \sum^d_{i,j=1} \lambda^{(i)} \sqrt{\mathrm{Var}(\mathbb{E}^W[(W_i'-W_i)(W_j'-W_j)])},
 \\ B & = \sum^d_{i,j,k=1} \lambda^{(i)} \mathbb{E}[|(W_i'-W_i)(W_j'-W_j)(W_k'-W_k)|], \quad C  = \sum^d_{i=1} \lambda^{(i)} \sqrt{\mathrm{Var} (R_i)}.
\end{split}
\end{equation*}
\end{theorem}

The following theorem provides a bound under a weaker assumption on the test function $h.$     
  
\begin{theorem}  \label{prop2.9}
Suppose the same assumptions as in Theorem \ref{prop2.7} hold, except that we now only assume that $h$ is twice differentiable. Let $\sigma_*= \min_{1\leq i \leq d} [ \sum^d_{j=1} \tilde{\sigma}_{ij}^2]^{\frac{1}{2}}$, where $\tilde{\sigma}_{ij} = (\Sigma^{-\frac{1}{2}})_{ij}. $ Then
\begin{equation*}
|\mathbb{E}[h(W)]-\mathbb{E}[h(Y)]| \leq \frac{|h|_2}{4}A+\frac{\sqrt{2\pi}\sigma_* |h|_2}{16} B +\bigg(|h|_1 + \frac{1}{2}d \|\Sigma \|^{\frac{1}{2}} |h|_2 \bigg)C.
\end{equation*}
\end{theorem}

\begin{remark}
Theorem \ref{prop2.9} is easily derived with only a minor change to the argument used by \cite{Reinert Multivariate} to prove Theorem \ref{prop2.7}. However, the result could be quite useful because it provides a bound on the quantity $|\mathbb{E}[h(W)]-\mathbb{E}[h(Y)]|$ that is almost the same as that of  Theorem \ref{prop2.7} but holding for a larger class of test functions. The only cost is that the covariance matrix $\Sigma$ must be invertible and that an extra factor $\sigma_*$ appears in the bound, although it is fairly simple and can also be
bounded  by $\sigma_*\leq\sqrt{d}\|\Sigma^{-\frac{1}{2}}\|$.
\end{remark}

\begin{corollary}\label{prop2.6} Suppose that $\sigma_{jj} = (\Sigma)_{jj} >0 $ for all $1 \leq j \leq d$, and set $\sigma = \min_{1\leq j \leq d} \sigma_{jj}. $  Suppose all other notation and assumptions of Theorems \ref{prop2.7} and \ref{prop2.9} hold. Then,
if  $d_3(W,Y) \leq \frac{2+\sqrt{2 \log d}}{2\sigma}$ and $d_2(W,Y) \leq \frac{4+2\sqrt{2 \log d}}{\sigma}$, respectively,
\begin{align}\label{eq:11}
  d_{\mathrm{K}}(W,Y) &\leq  2\sqrt{2}\bigg(\frac{4+2\sqrt{2 \log d}}{\sigma}\bigg)^{\frac{3}{4}}\bigg(\frac{A}{4}+\frac{B}{12}+\bigg(1+\frac{1}{2}d \|\Sigma \|^{\frac{1}{2}}\bigg)C\bigg)^{\frac{1}{4}}, \\
\label{eq:12}d_{\mathrm{K}}(W,Y)& \leq  2\bigg(\frac{4+2\sqrt{2 \log d}}{\sigma}\bigg)^{\frac{2}{3}}\bigg(\frac{A}{4}+\frac{\sqrt{2\pi}\sigma_*B}{16}+\bigg(1+\frac{1}{2}d \|\Sigma \|^{\frac{1}{2}}\bigg)C\bigg)^{\frac{1}{3}}.
\end{align}
\end{corollary}

\begin{remark}If $A$, $B$ and C are order $n^{-\frac{1}{2}},$ then the bounds (\ref{eq:11}) and (\ref{eq:12}) are of order $n^{-\frac{1}{8}}$ and $n^{-\frac{1}{6}}$, respectively. The bounds (\ref{eq:11}) and (\ref{eq:12}) allow one to immediately obtain (sub-optimal order) Kolmogorov-distance bounds following an application of Theorem \ref{prop2.9} or the widely-used Theorem \ref{prop2.7} of \cite{Reinert Multivariate}. Corollary 3.1 of \cite{Reinert Multivariate} gives a bound, without explicit dependence on the dimension $d$, for non-smooth test functions in terms of the quantities $A'$, $B'$ and $C'$ that are more complicated than $A,$ $B$ and $C.$ If $A'$, $B'$ and $C'$ are of order $n^{-\frac{1}{2}},$ then the bound of \cite{Reinert Multivariate} is of order $n^{-\frac{1}{4}}.$ \rg{Therefore, whilst ours is not the first Kolmogorov distance bound involving exchangeable pairs couplings for multivariate normal approximation, it does have a benefit over the other available bound by being completely explicit and typically simpler to use in applications.}  We also remark that \cite{rr96} obtained order $n^{-\frac{1}{2}}\log n$ bounds \rg{to quantify multivariate normal approximation under local dependence assumptions}, although this rate was obtained by assuming that the random vectors are bounded. \rg{Optimal order $n^{-1/2}$ bounds for the multivariate central limit theorem have also been achieved by Stein's method \cite{bh10,gotze,raic}, and when specialised to this case our bound is of course not competitive.}
\end{remark}

 \noindent{\emph{Proof of Theorem \ref{prop2.9}.}} 
From inequality (2.11) of \cite{Reinert Multivariate}, we have that 
\begin{equation}\label{eq:8}
|\mathbb{E}[h(W)]-\mathbb{E}[h(Y)]|  \leq \frac{|h|_2}{4}A+\frac{B}{4} \max_{1\leq i,j,k\leq d}\bigg\| \frac{\partial^3 f_h(w)}{\partial w_i \partial w_j \partial w_k} \bigg\| +\bigg(|h|_1 + \frac{1}{2}d \|\Sigma \|^{\frac{1}{2}} |h|_2 \bigg)C,
\end{equation} 
where $f_h$ is the solution of the multivariate normal Stein equation (given in \cite{Reinert Multivariate}). By inequality (2.2) of \cite{gaunt gaunt rates} we have the bound $|f_h|_3\leq \frac{\sqrt{2\pi}\sigma_*}{4}|h|_2$,
and applying this inequality
into (\ref{eq:8}) yields the desired bound; \cite{Reinert Multivariate} used the inequality $|f_h|_3\leq\frac{1}{3}|h|_3$ instead.	
\hfill $\Box$

\vspace{3mm}

\noindent{\emph{Proof of Corollary \ref{prop2.6}.}} Set $|h|_1=|h|_2=|h|_3=1$ in the bounds of Theorems \ref{prop2.7} and \ref{prop2.9} to get bounds for $d_3(W,Y)$ and $d_2(W,Y)$, respectively.
Then apply Proposition \ref{prop2.5} with $m=3$ and $m=2$ (so that $N_{3}'=32$ and $N_{2}'=4$) to get the desired bounds.
\hfill $\Box$


\subsection{Beta approximation}\label{sec3.3}
The beta distribution provides an example of a distribution with possibly more than one singularity, depending on the parameter values. 

\begin{example} \label{beta example}
Let $X$ be an arbitrary random variable and let $Y$ follow the beta distribution with parameters $\alpha,\beta>0$, denoted by $\mathrm{Beta}(\alpha,\beta).$ The density function of $Y$ is $p(y) =\frac{ y^{\alpha-1}(1-y)^{\beta-1}}{B(\alpha,\beta)} $, $y \in (0,1)$, where $B(\alpha,\beta)$ is the beta function. 

\vspace{2mm}

\noindent{(i)}: First, suppose that $\alpha,\beta \geq 1.$ The density is bounded, and elementary calculus shows that $p(y)$ attains its maximum at $ y = \frac{\alpha-1}{\alpha+\beta-2} $ for $\alpha,\beta \geq 1$ (we exclude the case $\alpha=\beta=1$, in which case $p(y)=1$ for all $y\in(0,1)$), and therefore $p(y) \leq A_{\alpha,\beta} $, where  
\begin{equation*}
A_{\alpha,\beta}   = \frac{1}{B(\alpha,\beta)}\frac{(\alpha-1)^{\alpha-1}(\beta-1)^{\beta-1}}{(\alpha+\beta-2)^{\alpha+\beta-2}},
\end{equation*}
and we set $A_{1,1}=1$. Using case (i) of Proposition \ref{prop2.1} gives that, for $d_m(X,Y) \leq \frac{A_{\alpha,\beta}}{2N_{m}}$,
\begin{equation*}
d_{\mathrm{K}}(X,Y) \leq 2\big(A_{\alpha,\beta}^mM_md_m(X,Y)\big)^{\frac{1}{m+1}}.
\end{equation*}


\noindent{(ii):} Now suppose exactly one of $\alpha,\beta \in (0,1)$. Then there exists only one singularity point for this density function, and $p(y) \leq \frac{1}{B(\alpha,\beta)}|y|^{-(1-\alpha)}$ if  $\alpha<\beta$, or $p(y) \leq \frac{1}{B(\alpha,\beta)} |1-y|^{-(1-\beta)}$ if $\beta<\alpha$.
Using case (iii) of Proposition \ref{prop2.1} gives that, for $d_m(X,Y) \leq \frac{2^{1-\min(\alpha,\beta)}}{N_{m}B(\alpha,\beta) \cdot \min(\alpha,\beta) }$,
\begin{equation*}
d_{\mathrm{K}}(X,Y) \leq 2\bigg(\frac{2^{1-\min(\alpha,\beta)}}{B(\alpha,\beta)\cdot \min(\alpha,\beta)}\bigg)^{\frac{m}{m+\min(\alpha,\beta)}}\big(N_md_m(X,Y)\big)^{\frac{\min(\alpha,\beta)}{m+\min(\alpha,\beta)}}.
\end{equation*}


\noindent{(iii):} Finally, suppose both $\alpha,\beta \in (0,1).$ Then $p(y)\leq \frac{2^{1-\beta}}{B(\alpha,\beta)}|y|^{-(1-\alpha)}$ if $y \in (0,\frac{1}{2})$, and $p(y)\leq \frac{2^{1-\alpha}}{B(\alpha,\beta)}|1-y|^{-(1-\beta)}$ if $y \in (\frac{1}{2},1)$.
At the singularity points $y_1=0$, $y_2=1$, we have $p(y) \leq A|y-y_i|^{-a} $ for $|y-y_1|,|y-y_2| < \frac{1}{2}, $ where $A = \frac{2^{1-\min(\alpha,\beta)}}{B(\alpha,\beta)}$ and $a= 1-\min(\alpha,\beta).$ If we assume that $d_m(X,Y) \leq \frac{4^{1-\min (\alpha,\beta)}}{N_{m}B(\alpha,\beta) \cdot \min (\alpha,\beta)} $, then, by case (iii) of Proposition \ref{prop2.1}, 
\begin{equation*}
 d_{\mathrm{K}}(X,Y)  \leq  2\bigg(\frac{4^{1-\min(\alpha,\beta)}}{B(\alpha,\beta)\cdot \min(\alpha,\beta)}\bigg)^{\frac{m}{m+\min(\alpha,\beta)}}\big(N_md_m(X,Y)\big)^{\frac{\min(\alpha,\beta)}{m+\min(\alpha,\beta)}}.
\end{equation*}


\noindent{(iv):} We observe that the following bound applies for any parameter values $\alpha,\beta>0.$ Suppose $d_m(X,Y) \leq \frac{2^{-\min(\alpha,\beta,1)}}{N_mB(\alpha,\beta)\cdot \min(\alpha,\beta,1)}. $ Then 
\begin{equation}\label{eq:13}
d_{\mathrm{K}}(X,Y) \leq 2\bigg(\frac{4^{1-\min(\alpha,\beta,1)}}{B(\alpha,\beta) \cdot \min(\alpha,\beta,1)}\bigg)^{\frac{m}{m+\min(\alpha,\beta,1)}}\big(N_md_m(X,Y)\big)^{\frac{\min(\alpha,\beta,1)}{m+\min(\alpha,\beta,1)}}.
\end{equation}
\end{example}

\begin{example}Let us now see how the bounds of Example \ref{beta example} can be applied to a distributional approximation concerning the classical P\'olya-Eggenberger urn model. At time zero, we have $\alpha \geq 1 $ white balls and $\beta \geq 1$ black balls. A ball is chosen from the urn independently at every integer time and replaced along with $t \geq 1$ additional balls of the same colour. Let $S_n = S_n^{\alpha,\beta,n} $ denote the number of white balls drawn from the urn at time $n=0,1,\ldots$. Define $W_n=\frac{S_n}{n}$. It is a classical result that, as $n \rightarrow \infty$, $\mathcal{L}(W_n) \rightarrow_d \mathrm{Beta} (\frac{\alpha}{t},\frac{\beta}{t}).$ Let $Y \sim \mathrm{Beta} (\frac{\alpha}{t},\frac{\beta}{t})$. Theorem 1.1 of \cite{gr13} provides a (optimal order) Wasserstein distance bound for this distributional approximation: 
\begin{equation}\label{eq:14}
d_{\mathrm{W}}(W_n, Y ) \leq \bigg(\frac{t+\min(\alpha,\beta)}{2nt}+\frac{\alpha\beta}{nt(\alpha+\beta)}\bigg)(b_0+b_1)+\frac{3}{2n}:=\frac{C(\alpha,\beta,t)}{n},
\end{equation}
where $b_0=b_0(\frac{\alpha}{t},\frac{\beta}{t})$ and $b_1=b_1(\frac{\alpha}{t},\frac{\beta}{t})$ are some constants defined in \cite{gr13}. Combining (\ref{eq:13}) and (\ref{eq:14}) yields the Kolmogorov distance bound
\begin{equation*}
d_{\mathrm{K}}(X,Y) \leq 2\bigg(\frac{4^{1-\min(\frac{\alpha}{t},\frac{\beta}{t},1)}}{B(\frac{\alpha}{t},\frac{\beta}{t}) \cdot \min(\frac{\alpha}{t},\frac{\beta}{t},1)}\bigg)^{\frac{1}{1+\min(\alpha/t,\beta/t,1)}} \bigg(\frac{C(\alpha,\beta,t)}{n}\bigg)^{\frac{\min(\alpha/t,\beta/t,1)}{1+\min(\alpha/t,\beta/t,1)}},
\end{equation*}
which holds provided $\frac{C(\alpha,\beta,t)}{n} \leq \frac{2^{-\min(\alpha/t,\beta/t,1)}}{B(\alpha/t,\beta/t)\cdot \min(\alpha/t,\beta/t,1 )}. $  \rg{This is the first Kolmogorov distance bound for this particular distributional approximation concerning the classical P\'olya-Eggenberger urn model; the works of \cite{gr13} and \cite{d15}, which also considered this distributional approximation, worked in the Wassrstein and $d_{[2]}$ metrics, respectively.}

\rg{In \cite{d15}, a bound is given for the distance between the distributions of $W_n$ and $Y$ with respect to the $d_{[2]}$ metric that has the same rate of convergence with respect to $n$ as that of the Wasserstein distance bound (\ref{eq:14}). We are also able to use Proposition \ref{prop2.1} to derive a Kolmogorov distance bound in this case. However, as the bound is given with respect to the $d_{[2]}$ metric, the resulting Kolmogorov distance bound would have a slower order $n^{-\frac{\min(\alpha/t,\beta/t,1)}{2+\min(\alpha/t,\beta/t,1)}}$ rate of convergence.}
\end{example}

\subsection{Variance-gamma approximation}\label{sec3.5}

The variance-gamma (VG) distribution provides an example of a distribution whose density is either bounded or has a logarithmic or power law singularity depending on the parameter values. The VG distribution with parameters $r > 0$, $\theta \in \mathbb{R}$, $\sigma >0$, $\mu \in \mathbb{R}$, denoted by $\mathrm{VG}(r,\theta,\sigma,\mu)$, has probability density function
\begin{equation}\label{vgdefn}p(y) = \frac{1}{\sigma\sqrt{\pi} \Gamma(\frac{r}{2})} \mathrm{e}^{\frac{\theta}{\sigma^2} (y-\mu)} \bigg(\frac{|y-\mu|}{2\sqrt{\theta^2 +  \sigma^2}}\bigg)^{\frac{r-1}{2}} K_{\frac{r-1}{2}}\bigg(\frac{\sqrt{\theta^2 + \sigma^2}}{\sigma^2} |y-\mu| \bigg), 
\end{equation}
with support $\mathbb{R}$.  In the limit $\sigma\rightarrow 0$ the support becomes the interval $(\mu,\infty)$ if $\theta>0$, and is $(-\infty,\mu)$ if $\theta<0$. Here $K_\nu(y)$ is a modified Bessel function of the second kind. Basic properties of the VG distribution are given in \cite{gaunt vg,kkp01}. Upper bounds on $d_{\mathrm{K}}(X,Y)$ for an arbitrary random variable $X$ and $Y\sim\mathrm{VG}(r,\theta,\sigma,\mu)$ in terms of $d_{[1]}(X,Y)=d_{\mathrm{W}}(X,Y)$ were given by \cite{gaunt vg2} for the case $\theta=0$ and by \cite{gaunt vg3} for general $\theta\in\mathbb{R}$. We now apply Proposition \ref{prop2.1} to provide an upper bound on $d_{\mathrm{K}}(X,Y)$ in terms of $d_m(X,Y)$, $m\geq1$.

\begin{example}\label{exvg1}
Throughout this example, we work in the general case $r > 0$, $\theta \in \mathbb{R}$, $\sigma >0$, $\mu \in \mathbb{R}$.

\vspace{2mm}

\noindent{(i):} Suppose first that $r>1$. In this case, the density is bounded. Moreover, we have the bound $p(y)\leq A_{r,\theta,\sigma}$ for all $y\in\mathbb{R}$, where
\begin{align*}A_{r,\theta,\sigma}= \left\{
\begin{aligned}
&\frac{\Gamma(\frac{r-1}{2})}{2\sigma\sqrt{\pi} \Gamma(\frac{r}{2})}\bigg(\frac{\sigma^2}{\theta^2+\sigma^2}\bigg)^{\frac{r-1}{2}}, & & 1<r\leq2, \\
&\frac{\Gamma(\frac{r-1}{2})}{2\sigma\sqrt{\pi} \Gamma(\frac{r}{2})}\bigg(\frac{\sigma^2}{\theta^2+\sigma^2}\bigg)^{\frac{r-1}{2}} \mathrm{e}^{\frac{\theta^2}{\sigma^2} (r-2)}, & & r>2
\end{aligned}
\right.
\end{align*}
(see the proof of Proposition 3.6 of \cite{gaunt vg3}, and \cite[Proposition 2.1]{gaunt vg3}). Applying part (i) of Proposition \ref{prop2.1} now gives that, if $d_m(X,Y)\leq\frac{A_{r,\theta,\sigma}}{2N_m}$,
\begin{equation}\label{vgbd1}d_{\mathrm{K}}(X,Y)\leq2\big(A_{r,\theta,\sigma}^mM_md_m(X,Y)\big)^{\frac{1}{m+1}}.
\end{equation}


\noindent{(ii):} Now suppose $r=1$.  We begin by noting the inequality $\mathrm{e}^{\beta x}K_0(|x|)<-\pi\log|x|$, for $|x|\leq0.645$, $|\beta|\leq1$ (which follows easily from \cite[Lemma 5.1, (ii)]{gaunt vg3}). Then, for $|y-\mu|\leq\frac{0.645\sigma^2}{\sqrt{\theta^2+\sigma^2}}$,
\begin{align*}p(y)&=\frac{1}{\pi\sigma}\mathrm{e}^{\frac{\theta}{\sigma^2} (y-\mu)}K_0\bigg(\frac{\sqrt{\theta^2+\sigma^2}}{\sigma^2}|y-\mu|\bigg)<-\frac{1}{\sigma}\log\bigg(\frac{\sqrt{\theta^2+\sigma^2}}{\sigma^2}|y-\mu|\bigg).
\end{align*}
Applying part (ii) of Proposition \ref{prop2.1} with $A=\frac{1}{\sigma}$ and $c=\frac{\sqrt{\theta^2+\sigma^2}}{\sigma^2}$ yields that, for $d_m(X,Y)\leq\frac{1}{\sigma N_m}\min\big(1,(\frac{1.29\sigma^2}{\sqrt{\theta^2+\sigma^2}})^{m+1})$,
\begin{align}\label{vgbd2}d_{\mathrm{K}}(X,Y)\leq\bigg[2+\frac{1}{m+1}\log\bigg(\frac{2\sigma^{2m+1}}{M_m(\theta^2+\sigma^2)^{\frac{m+1}{2}}d_m(X,Y)}\bigg)\bigg]\bigg(\frac{N_m}{\sigma^m}d_{m}(X,Y)\bigg)^\frac{1}{m+1}.
\end{align}


\noindent{(iii):} Finally, we suppose $0<r<1$.
Part (i) of Lemma 5.1 of \cite{gaunt vg3} tells us that $e^xx^\mu K_\mu(x)\leq 2^{\mu-1}\Gamma(\mu)$, for $x>0$, $0<\mu\leq\frac{1}{2}$. But also, $K_\nu(x)=K_{-\nu}(x)$ for all $x\in\mathbb{R}$ (see \cite{olver}), and so $K_{\nu}(|x|)\leq 2^{-\nu-1}\Gamma(-\nu)e^{-|x|}|x|^\nu$, for $x\in\mathbb{R}$, $-\frac{1}{2}<\nu<0$. Applying this inequality with $\nu=\frac{r-1}{2}$, we deduce that, for $y\in\mathbb{R}$,
\begin{align*}p(y)\leq \frac{\Gamma(\frac{1-r}{2})|y-\mu|^{r-1}}{(2\sigma)^{r}\sqrt{\pi} \Gamma(\frac{r}{2}) }e^{\frac{\theta}{\sigma^2}(y-\mu)} e^{-\frac{\sqrt{\theta^2+\sigma^2}}{\sigma^2}|y-\mu|} \leq \frac{\Gamma(\frac{1-r}{2}) |y-\mu|^{-(1-r)}}{(2\sigma)^{r}\sqrt{\pi} \Gamma(\frac{r}{2}) }.
\end{align*}
Applying part (iii) of Proposition \ref{prop2.1} with $A=\frac{\Gamma(\frac{1-r}{2})}{(2\sigma)^{r}\sqrt{\pi} \Gamma(\frac{r}{2}) } $ and $a=1-r$ gives that, for $d_m(X,Y)\leq\frac{\Gamma(\frac{1-r}{2})}{N_m(4\sigma)^r\sqrt{\pi}\Gamma(\frac{r}{2}+1)}$,
\begin{equation}\label{vgbd3}
d_{\mathrm{K}}(X,Y) \leq 2\bigg(\frac{\Gamma(\frac{1-r}{2})}{(4\sigma)^r\sqrt{\pi}\Gamma(\frac{r}{2}+1)}\bigg)^{\frac{m}{m+r}}\big(N_md_m(X,Y)\big)^{\frac{r}{m+r}}.
\end{equation}
\end{example}

\begin{example} Suppose $r > 0$, $\theta \in \mathbb{R}$, $\sigma >0$, and let us write $\mathrm{VG}_c(r,\theta,\sigma)$ for $\mathrm{VG}(r,\theta,\sigma,-r\theta)$, so that $Y\sim\mathrm{VG}_c(r,\theta,\sigma)$ has zero mean (see \cite[equation (2.3)]{gaunt vg}). Let $(F_n)_{ n \geq 1 }$ be a sequence of elements from the second Wiener chaos. 
Define $\mathbf{M}(F_n) = \max \left\{ |\kappa_\ell(F_n) - \kappa_\ell(Y) |: \ell = 2,\ldots,6 \right\}. $
 Formulas for the first six cumulants of  $Y\sim\mathrm{VG}_c(r,\theta,\sigma)$ are given by \cite{eichelsbacher}. Then, a beautiful result of \cite{aet21} provides the following quantitative ``six moment" theorem for the VG approximation of elements from the second Wiener chaos with optimal convergence rate: there exist constants $C_1,C_2>0$, depending only on $r$, $\theta$ and $\sigma$, such that
  \begin{equation}\label{optimalvg}C_1\mathbf{M}(F_n)\leq d_{[2]}(F_n,Y)\leq C_2\mathbf{M}(F_n).
  \end{equation}
This is a VG analogue of the optimal four moment theorems for normal and centered gamma approximation of \cite{np15} and \cite{aek20}, respectively.
  
Applying the bounds (\ref{vgbd1})--(\ref{vgbd3}) of Example \ref{exvg1} (with $m=2$) to the upper bound of (\ref{optimalvg}) yields a quantitative six-moment theorem for VG approximation with respect to the Kolmogorov metric: there exist constants $C_3,C_4,C_5,C_6>0$, depending only on $r$, $\theta$ and $\sigma$, such that    
\begin{align}\label{optk}d_{\mathrm{K}}(F_n,Y)\leq \left\{\begin{aligned}
&C_3\big(\mathbf{M}(F_n)\big)^{\frac{1}{3}}, & & r>1, \\
&C_4\big(\mathbf{M}(F_n)\big)^{\frac{1}{3}}\log\bigg(\frac{C_5}{\mathbf{M}(F_n)}\bigg), & & r=1, \\
&C_6\big(\mathbf{M}(F_n)\big)^{\frac{r}{2+r}}, & & 0<r<1.
\end{aligned}
\right.
\end{align}
\rg{The estimates in (\ref{optk}) are the first Kolmogorov distance estimates in the literature for quantifying the six-moment theorem for VG approximation. Whilst the rate is likely to be suboptimal for each of the three cases considered ($r>1$, $r=1$ and $0<r<1$), the benefit of these estimates is that due to technical difficulties involving the solution of the VG Stein equation it seems very difficult to directly apply the Malliavin-Stein method to obtain Kolmogorov distance bounds for this distributional approximation; a discussion of these technical difficulties is given in Remark 5.5 of \cite{gaunt vg2}.} 
\end{example}

\begin{example} Consider the generalized Rosenblatt process $Z_{\gamma_1,\gamma_2}(t)$, introduced by \cite{mt12} as the double Wiener-It\^{o} integral 
\begin{equation*}Z_{\gamma_1,\gamma_2}(t)=\int_{\mathbb{R}^2}^{\prime}\bigg(\int_0^t(s-x_1)_+^{\gamma_1}(s-x_2)_+^{\gamma_2}\,ds\bigg)\,dB_{x_1}\,dB_{x_2},
\end{equation*}
where the prime $\prime$ indicates exclusion of the diagonals $x_1=x_2$ in the stochastic integral, $B_{x}$ is standard Brownian motion and $\gamma_i\in(-1,-\frac{1}{2})$, $i=1,2$, and $\gamma_1+\gamma_2>-\frac{3}{2}$. 
Since $Z_{\gamma_1,\gamma_2}(t)=_d t^{2+\gamma_1+\gamma_2}Z_{\gamma_1,\gamma_2}(1)$, we will work with the random variable $Z_{\gamma_1,\gamma_2}(1)$; results for general $t>0$ follow from rescaling.

Recently, \cite{aet21} have obtained optimal convergence results for some remarkable limit theorems of \cite{bt17} concerning the random variable $Z_{\gamma_1,\gamma_2}(1)$ at critical values of $\gamma_1$ and $\gamma_2$. 

\begin{itemize}

\item[(a)] Suppose $(\gamma_1,\gamma_2)\rightarrow(-\frac{1}{2},\gamma)$, where $-1<\gamma<-\frac{1}{2}$. Let $Y\sim\mathrm{VG}_c(1,0,1)$. Then, there exists constants $C_1,C_2>0$, such that, as $\gamma_1\rightarrow-\frac{1}{2}$,
\begin{align*}C_1\Big|-\gamma_1-\frac{1}{2}\Big|\leq d_{[2]}(Z_{\gamma_1,\gamma_2}(1),Y)\leq C_2\Big|-\gamma_1-\frac{1}{2}\Big|.
\end{align*}

\item[(b)] Let $\rho\in(0,1)$ and consider the random variable $Y_\rho\sim\mathrm{VG}_c(1,\frac{a_\rho-b_\rho}{\sqrt{2}},\sqrt{2a_pb_p})$, where
\begin{align*}a_\rho=\frac{(2\sqrt{\rho})^{-1}+(\rho+1)^{-1}}{\sqrt{(2\rho)^{-1}+2(\rho+1)^{-2}}}, \quad b_\rho=\frac{(2\sqrt{\rho})^{-1}-(\rho+1)^{-1}}{\sqrt{(2\rho)^{-1}+2(\rho+1)^{-2}}}.
\end{align*}
Suppose $\gamma_1\geq\gamma_2$ and that $\gamma_2=\frac{\gamma_1+1/2}{\rho}-\frac{1}{2}$. Then, there exists constants $C_3,C_4>0$, depending only on $\rho$, such that, as $\gamma_1\rightarrow-\frac{1}{2}$ (so that automatically $\gamma_2\rightarrow-\frac{1}{2}$),
\begin{align*}C_3\Big|-\gamma_1-\frac{1}{2}\Big|\leq d_{[2]}(Z_{\gamma_1,\gamma_2}(1),Y_\rho)\leq C_4\Big|-\gamma_1-\frac{1}{2}\Big|.
\end{align*}
A bound with slower convergence rate given with respect to the stronger 2-Wasserstein distance had previously been given by \cite{aaps17}.

\end{itemize}

In both cases (a) and (b) with corresponding target random variable $Y$, applying inequality (\ref{vgbd2}) of Example \ref{exvg1} yields that there exists a constant $C_5>0$ such that, as $\gamma_1\rightarrow-\frac{1}{2}$,
\[d_{\mathrm{K}}(Z_{\gamma_1,\gamma_2}(1),Y)\leq C_5\bigg|-\gamma_1-\frac{1}{2}\bigg|^{\frac{1}{3}}\log\bigg|\frac{1}{-\gamma_1-\frac{1}{2}}\bigg|,\]
improving on the rate of convergence of a recent Kolmogorov distance bound of \cite{gaunt vg3}.
\end{example}

\section{Smooth approximations of indicator functions}\label{sec4}

In this section, we study smooth approximations to indicator functions. We focus on univariate indicator functions in Section \ref{sec4.1} and extend to multivariate indicator functions in Section \ref{sec4.2}. The main results of this section (bounds on the derivatives of the approximating functions, which are given in Corollary \ref{cormax} and Lemma \ref{prop2.13}) are used in Section \ref{sec5} to prove the results of Section \ref{sec2}.

\subsection{Univariate indicator functions}\label{sec4.1}
In this section, we seek piecewise-defined polynomials $h_{m,z,\alpha}(x)$ to approximate the indicator function $\mathbf{1}_{ x \leq z }$  for any $ m \geq 1 $, $z \in \mathbb{R}$ and $\alpha >0.$ Our target function $h_{m,z,\alpha}:\mathbb{R}\rightarrow[0,1]$ will be such that $h_{m,z,\alpha}(x) \geq \mathbf{1}_{x \leq z}$, with $h_{m,z,\alpha}(x)=1$ for $x \leq z$ and $ h_{m,z,\alpha}(x) = 0$ for $x \geq z+\alpha$, and its $(m-1)-$th derivative $h_{m,z,\alpha}^{(m-1)}(x)$ will be Lipschitz. We denote the class of such functions by $\mathcal{F}_{m,z,\alpha}$. The aim of this section is to construct such a function which also provides an accurate bound for the quantity $\max_{0 \leq i \leq m}\|h_{m,z,\alpha}^{(i)} \|$, a term that arises in our proofs of the results of Section \ref{sec2}.
For ease of notation, we will first work in the case $z=-1$, $\alpha=2$ (we set $h_m(x) :=h_{m,-1,2}(x) $); we then recover the general case $z \in \mathbb{R}$, $\alpha>0$ via the relation $h_{m,z,\alpha} (x) = h_m(\frac{2}{\alpha} (x-(z+\frac{\alpha}{2})) ). $ 

The function $h_m : \mathbb{R} \rightarrow [0,1] $ will satisfy:
\begin{itemize}
\item[(C1):] $h_m(x) = 1 $ for $x \in (-\infty,-1]$, $h_m(x)\in[0,1]$ for $x\in(-1,1)$, and $ h_m(x) = 0 $ for $x \in [1,\infty)$;
\item[(C2):] $h_m\in C^{m-1}(\mathbb{R})$ and $h_m^{(m-1)} $ is Lipschitz.
\end{itemize}
For a given $m\geq1$, the class of all functions satisfying conditions (C1) and (C2) will be denoted by $\mathcal{F}_m$. 

\begin{lemma}\label{simplelem}Let $m\geq1$. For $h_m\in\mathcal{F}_m$, we have $1=\|h_m\|\leq 2\|h_m'\|$, and $\| h_m^{(i)} \| \leq \| h_m^{(i+1)} \| $ for all $ 1 \leq i  \leq m-1. $ 
\end{lemma}

\begin{proof}
It is clear that $2\|h_m'\|\geq1=\|h_m\|$. Also, conditions (C1) and (C2) imply that $h_m^{(i)}(-1)=0$ for $1\leq i\leq m-1$, and so, by the mean value theorem, for any $-1 \leq x \leq 0,$
\begin{equation*}
  |h_m^{(i)}(x)|  = |h_m^{(i)}(x) - h_m^{(i)}(-1) |
   \leq (x+1) \|h_m^{(i+1)}\|
   \leq \|h_m^{(i+1)}\|.
\end{equation*}
We argue similarly for $0<x\leq1$ to deduce that $\|h_m^{(i)} \| \leq \|h_m^{(i+1)} \|$ for $1\leq i\leq m-1$.
\end{proof}

The proof of the following lemma is deferred to the end of this section.

\begin{lemma}\label{conj1}For $m\geq1$, there exists $h_m\in\mathcal{F}_m$ such that 
\begin{equation}\label{hmfor}\|h_m^{(m)}\|=2^{m-2}(m-1)!.
\end{equation}
\end{lemma}

\begin{corollary}\label{cormax}For $m\geq1$, $z\in\mathbb{R}$, $\alpha>0$, there exists $h_{m,z,\alpha}\in\mathcal{F}_{m,z,\alpha}$ such that
\begin{align}\label{nfbd1}\max_{0\leq i\leq m}\|h_{m,z,\alpha}^{(i)}\|&\leq2^{m-2}(m-1)!\bigg[1+\frac{2^m}{\alpha^m}\bigg], \quad \alpha>0,\\
\label{nfbd2}\max_{0\leq i\leq m}\|h_{m,z,\alpha}^{(i)}\|&\leq\frac{2^{2m-2}(m-1)!}{\alpha^m}, \quad 0<\alpha\leq1.
\end{align}
\end{corollary}

\begin{proof}We have $\|h^{(i)}_{m,z,\alpha} \| = (\frac{2}{\alpha})^i \|h_m^{(i)} \| $ for $0 \leq i \leq m$, and by Lemma \ref{simplelem} we obtain that
\begin{equation}\label{maxfor0}\max_{0 \leq i \leq m} \| h_{m,z,\alpha}^{(i)} \|\leq\|h_m^{(m)}\|\max_{0 \leq i \leq m}\bigg(\frac{2}{\alpha}\bigg)^i\leq \|h_m^{(m)}\|\bigg[1+\bigg(\frac{2}{\alpha}\bigg)^m\bigg].
\end{equation}
If we further assume that $\alpha \leq 1$, then we can obtain the bound
\begin{equation}\label{maxfor}
\max_{0 \leq i \leq m} \| h_{m,z,\alpha}^{(i)} \| \leq \|h_m^{(m)} \| \bigg(\frac{2}{\alpha}\bigg)^m.
\end{equation}
 Combining (\ref{maxfor0}) and (\ref{maxfor}) with (\ref{hmfor}) yields inequalities (\ref{nfbd1}) and (\ref{nfbd2}), respectively. 
\end{proof}

\begin{remark}(1): The restriction $\alpha\leq1$ in Corollary \ref{cormax} is made to yield a compact bound for $\max_{0\leq i\leq m}\|h_{m,z,\alpha}^{(i)}\|$, which will in turn yield compact bounds in our general results of Section \ref{sec2}.

\vspace{2mm}

\noindent{(2):} Our construction of $h_m\in\mathcal{F}_m$ in the proof of Lemma \ref{conj1} satisfies $|h_m^{(m)}(x)|=2^{m-2}(m-1)!$ for all $x\in(-1,1)$, that is it is constant on the whole interval; in the language of approximation theory such functions are referred to as
 perfect splines \cite{powell}. The function $h_m$ was carefully constructed to possess this feature in order to minimise the quantity $\|h_m^{(m)}\|$.
\end{remark}

\noindent{\emph{Proof of Lemma \ref{conj1}.}}
We first prove the result in detail for the case $m\geq2$ is even; the odd $m\geq1$ case is similar and we only highlight the points at which the argument differs. Let $m=2n\geq2$. Consider the function $h_m:\mathbb{R}\rightarrow\mathbb{R}$ defined by $h_m(x) = 1 $ for $x \in (-\infty,-1]$, $ h_m(x) = 0 $ for $x \in [1,\infty)$, and in the interval $x\in(-1,1)$ is given by
\begin{align*}
 h_m(x)=
 \left\{ \begin{aligned} 
 1- \frac{2^{m-2}}{m}\bigg[(1+x)^m-2\sum_{j=1}^k(-1)^{j+1}(x_j+x)^m\bigg], & & \text{$x \in -I_k$, $0\leq k\leq n-1$,} &
 \\ \frac{2^{m-2}}{m}\bigg[(1-x)^m-2\sum_{j=1}^k(-1)^{j+1}(x_j-x)^m\bigg], & & \text{$x \in I_k$, $0\leq k\leq n-1$,} &
 \end{aligned} \right.
 \end{align*} 
where $I_{k}=(x_{k+1},x_k]$ and $-I_k=(-x_k,-x_{k+1}]$ for $0\leq k\leq n-1$.
Here $x_k=\cos(\frac{\pi k}{m})$, and in order to give a more economical formula we have used the convention that $\sum_{j=1}^0a_j=0$. In interpreting the above formula, observe that $x_0=1$ and $x_n=0$.

We now prove that $h_m\in\mathcal{F}_m$; it will then be clear that $\|h_m^{(m)}\|=2^{m-2}(m-1)!$. For ease of exposition, in what follows we assume $m=2n\geq4$; the case $m=2n=2$ follows the same lines but is much simpler. We readily see that $h_m'(x)<0$ for $x\in(-1,1)$, meaning that $h(x)\in[0,1]$ for all $x\in\mathbb{R}$, and so condition (C1) is met. By construction, $h_m$ is clearly $(m-1)$-times differentiable with its $(m-1)$-th derivative being Lipschitz for all $x\in\mathbb{R}$, except possibly at the point $x=x_n=0$. We now prove that this is also the case at $x=0$. We need to prove that $h_m^{(i)}(0-)=h_m^{(i)}(0+)$ for $0\leq i\leq m-1$. Observe that $h_m(x)+h_m(-x)=1$ for all $x\in\mathbb{R}$. For $h_m$ to be continuous at $x=0$, we therefore require that $h_m(0-)=h_m(0+)=\frac{1}{2}$. This holds due to the formula $\sum_{j=1}^{n-1}(-1)^{j+1}\cos^{2n}(\frac{\pi j}{2n})=\frac{1}{2}-\frac{n}{2^{2n-1}}$ (see \cite{fonseca}). We also see that 
for odd $i\geq1$ we have $h_m^{(i)}(x)=h_m^{(i)}(-x)$ for all $x\in\mathbb{R}$, and so $h_m^{(i)}(0+)=h_m^{(i)}(0-)$ for odd $i$. Similarly, for even $i=2\ell\geq2$, we have that $h_m^{(2\ell)}(x)=-h_m^{(2\ell)}(-x)$, and for the derivatives $h_m^{(2\ell)}$, $1\leq\ell\leq n-1$, to be continuous at $x=0$ we require that $h_m^{(2\ell)}(0+)=h_m^{(2\ell)}(0-)=0$ for all $1\leq\ell\leq n-1$. This is indeed the case due to the formula $\sum_{j=1}^{n-1}(-1)^{j+1}\cos^{2\ell}(\frac{\pi j}{2n})=\frac{1}{2}$, $1\leq \ell\leq n-1$ (see \cite{fonseca}). We have thus verified condition (C2), and so have proved that $h_m\in\mathcal{F}_m$ for even $m\geq2$.

Let us now consider the case that $m=2n+1\geq1$ is odd. Consider the function $h_m:\mathbb{R}\rightarrow\mathbb{R}$ defined by $h_m(x) = 1 $ for $x \in (-\infty,-1]$, $ h_m(x) = 0 $ for $x \in [1,\infty)$, and in the interval $x\in(-1,1)$, for the case $n\geq1$, is given by
 \begin{align*}
 h_m(x)=
 \left\{ \begin{aligned} 
 1- \frac{2^{m-2}}{m}\bigg[(1+x)^m-2\sum_{j=1}^k(-1)^{j+1}(x_j+x)^m\bigg], & & \text{$x\in -I_k$, $0\leq k\leq n$,} & 
 \\ \frac{2^{m-2}}{m}\bigg[(1-x)^m-2\sum_{j=1}^k(-1)^{j+1}(x_j-x)^m\bigg], & & \text{$x \in I_k$, $0\leq k\leq n$,} &
 \end{aligned} \right.
 \end{align*}
 where $I_n=[0,x_n)$, $-I_{n}=(-x_n,0)$, and $I_k=[x_{k+1},x_{k})$ and $-I_k=(-x_k,-x_{k+1}]$ for $0\leq k\leq n-1$. Here $x_k=\cos(\frac{\pi k}{m})=\cos(\frac{\pi k}{2n+1})$.
For $m=1$ we define $h_1(x)=\frac{1}{2}-\frac{1}{2}x$ if $x\in(-1,1)$. 

The argument to verify that $h_m\in\mathcal{F}_m$ for odd $m\geq1$ is similar to the case of even $m\geq2$, and again doing so for the special case $m=1$ is simple. The only way the argument changes is that this time to confirm continuity of $h_m$ at $x=0$ we use the formula $\sum_{j=1}^{n}(-1)^{j+1}\cos^{2n+1}(\frac{\pi j}{2n+1})=\frac{1}{2}-\frac{2n+1}{2^{2n+1}}$, whilst we confirm that $h_m^{(2\ell)}(0+)=h_m^{(2\ell)}(0-)$, $1\leq \ell\leq n$, by using the formula $\sum_{j=1}^{n}(-1)^{j+1}\cos^{2\ell+1}(\frac{\pi j}{2n+1})=\frac{1}{2}$, $0\leq \ell\leq n-1$. These series involving the cosine function can be proved by making minor changes to the argument used to prove the two formulas of \cite{fonseca} that we quoted in proving that $h_m\in\mathcal{F}_m$ for even $m\geq2$. The proof is now complete.
\hfill $\Box$

\subsection{Multivariate indicator functions}\label{sec4.2}

Let $m\geq1$, $z\in\mathbb{R}^d$ and $\alpha>0$.
For $x=(x_1,\ldots,x_d)^T\in\mathbb{R}^d$ and $z=(z_1,\ldots,z_d)^T\in\mathbb{R}^d$, we write $x\leq z$ if $x_j\leq z_j$ for all $j=1,\ldots,d$.
In this section, we seek functions $h_{m,z,\alpha}:\mathbb{R}^d\rightarrow[0,1]$ such that $h_{m,z,\alpha}(x)\geq \mathbf{1}_{x\leq z}$, satisfying $h_{m,z,\alpha}(x)=1$ if $x\leq z$, $h_{m,z,\alpha}(x)=0$ if $x_j\geq z_j+\alpha$ for some $j\in\{1,\ldots,d\}$, and also with all $(m-1)$-th order partial derivatives being Lipschitz. We will denote this class of functions by $\mathcal{F}_{m,z,\alpha,d}$.






\begin{lemma}\label{prop2.13} 
For $d\geq1$, $m\geq1$, $z\in\mathbb{R}$ and $\alpha>0$, there exists $h_{m,z,\alpha}\in\mathcal{F}_{m,z,\alpha,d}$ such that
\begin{align}\label{multihbd0}\max_{0\leq i\leq m}|h_{m,z,\alpha}|_i&\leq 2^{\frac{3m}{2}-2}(m-1)!\bigg[1+\frac{2^m}{\alpha^m}\bigg],\quad \alpha>0,\\
\label{multihbd}\max_{0\leq i\leq m}|h_{m,z,\alpha}|_i&\leq \frac{2^{\frac{5m}{2}-2}(m-1)!}{\alpha^m}, \quad 0<\alpha\leq1.
\end{align}
where the notation $|h|_i$ is defined in Section \ref{sec2.2}. 
 If $1\leq m\leq 4$, then the bounds (\ref{multihbd0}) and (\ref{multihbd}) can be improved to
 \begin{align}\label{multihbd2yy}\max_{0\leq i\leq m}|h_{m,z,\alpha}|_i&\leq 2^{m-2}(m-1)!\bigg[1+\frac{2^m}{\alpha^m}\bigg], \quad \alpha>0, \\
\label{multihbd2}\max_{0\leq i\leq m}|h_{m,z,\alpha}|_i&\leq \frac{2^{2m-2}(m-1)!}{\alpha^m}, \quad 0<\alpha\leq1.
\end{align}
\end{lemma}

\begin{remark}We conjecture that inequality (\ref{multihbd2}) holds for all $m\geq1$. Our method of proving inequality (\ref{multihbd2}) for $1\leq m\leq4$ could also be applied to show that the inequality holds (assuming this is the case) for a given $m\in\{5,6,\ldots\}$, although the calculations would become more tedious as $m$ increases. 
\end{remark}

\begin{proof} Consider the function $h_{m,z,\alpha}:\mathbb{R}^d\rightarrow[0,1]$ defined by $h_{m,z,\alpha}(x) = \prod^d_{j=1} h_{m,z_j,\alpha}(x_j)$, where $h_{m,z_j,\alpha}\in\mathcal{F}_{m,z_j,\alpha}$, $1\leq j\leq d$, are the same functions as those given in the proof of Lemma \ref{conj1}, following an application of the formula $h_{m,z_j,\alpha} (x) = h_m(\frac{2}{\alpha} (x-(z_j+\frac{\alpha}{2})) ) $ . All of the $(m-1)$-th order partial derivatives of $h_{m,z,\alpha}$ are Lipschitz due to the standard result that the product of bounded Lipschitz functions is Lipschitz, and since $h_{m,z_j,\alpha}\in\mathcal{F}_{m,z_j,\alpha}$, $1\leq j\leq d$, we readily deduce that $h_{m,z,\alpha}\in\mathcal{F}_{m,z,\alpha,d}$. 

Arguing as we did in arriving at inequalities (\ref{maxfor0}) and (\ref{maxfor}) we obtain that
\begin{align*}
\max_{0 \leq i \leq m} | h_{m,z,\alpha} |_i &\leq |h_m|_m \bigg[1+\bigg(\frac{2}{\alpha}\bigg)^m\bigg], \quad \alpha>0, \\
\max_{0 \leq i \leq m} | h_{m,z,\alpha} |_i &\leq |h_m|_m \bigg(\frac{2}{\alpha}\bigg)^m, \quad 0<\alpha\leq1.
\end{align*} 
We have that
\begin{align}\label{prodfor}|h|_m=\max_{n_1,\ldots,n_d}\bigg\|\frac{\partial^{m}}{\partial x_1^{n_1} \ldots \partial x_d^{n_d}  }h_{m}(x) \bigg\|\leq\max_{n_1,\ldots,n_d}\prod_{j=1}^d\|h_{m}^{(n_j)}\|,
\end{align}
where the maximum is taken over all $n_1,\ldots,n_d\geq0$ such that $\sum_{j=1}^dn_j=m$. We can bound the product in (\ref{prodfor}) using the Landau-Kolmogorov inequality \cite{k49,l13} for functions defined on the real line, $\|g^{(k)}\|\leq C_{m,k}\|g\|^{1-\frac{k}{m}}\|g^{(m)}\|^{\frac{k}{m}}$, $1\leq k\leq m-1$, where $C_{m,k}=K_{m-k}K_m^{-1+\frac{k}{m}}$ and $K_r=\frac{4}{\pi}\sum_{j=0}^\infty\big[\frac{(-1)^j}{2j+1}\big]^{r+1}$ is the Favard constant of order $r$.
Since $\|h_{m}\|=1$ and $\|h_{m}^{(m)}\|=2^{m-2}(m-1)!$ we get that $\|h_{m}^{(k)}\|\leq 2^{k(m-2)/m} C_{m,k}((m-1)!)^{\frac{k}{m}}$ for $1\leq k\leq m-1$. Therefore, since $\sum_{j=1}^dn_j=m$,
\begin{align*}|h_m|_m\leq\max_{n_1,\ldots,n_d}\prod_{j=1}^d2^{n_j(m-2)/m} C_{m,n_j}((m-1)!)^{\frac{n_j}{m}}=C_{m}\cdot 2^{m-2}(m-1)!,
\end{align*} 
where
\begin{align}\label{mmfor}C_m=\max_{n_1,\ldots,n_d}\prod_{j=1}^dC_{m,n_j}=\max_{n_1,\ldots,n_d}\prod_{j=1}^dK_{m-n_j}K_m^{-1+\frac{n_j}{m}}=K_m\max_{n_1,\ldots,n_d}\prod_{j=1}^d\frac{K_{m-n_j}}{K_m}.
\end{align}
We now note that $1=K_0<K_2<K_4<\cdots<\frac{4}{\pi}<\cdots<K_3<K_1=\frac{\pi}{2}$ (see \cite{k91}).
Therefore, for $m\geq2$, we have $K_m\leq K_3=\frac{1}{24}\pi^3$ and $K_m\geq K_2=\frac{1}{8}\pi^2$, whilst for $m\geq2$ and $n_j\geq1$ we have $K_{m-n_j}\leq K_1=\frac{\pi}{2}$. Applying these inequalities to (\ref{mmfor})  yields the bound
\begin{align*}C_m\leq\frac{\pi^3}{24}\bigg(\frac{\pi/2}{\pi^2/8}\bigg)^m=\frac{\pi^3}{24}\bigg(\frac{4}{\pi}\bigg)^m\leq 2^{\frac{m}{2}}, \quad m\geq3.
\end{align*} 
We have thus proved inequalities (\ref{multihbd0}) and (\ref{multihbd}) for $m\geq3$, which suffices because (\ref{multihbd2yy}) and (\ref{multihbd2}) (proved below) provide more accurate bounds for $m=1,2$.

For $1\leq m\leq 4$ we can obtain the improved bounds (\ref{multihbd2yy}) and (\ref{multihbd2}) by obtaining a more accurate bound for the right-hand side of inequality (\ref{prodfor}). Let $h_1,h_2,h_3,h_4$ be the functions constructed in the proof of Lemma \ref{conj1}, which are listed in Appendix \ref{appcon}. Then direct calculations show that  
$\|h_1^{(1)}\|=\frac{1}{2}$, $\max(\|h_2^{(2)} \|, \|h_2^{(1)} \|^2 ) = \max(1,1) = 1$, $\max(\|h_3^{(3)} \|, \|h_3^{(2)} \|\cdot \|h_3^{(1)} \|,\|h^{(1)}_3 \|^3 ) = \max(4,2,1) = 4$ and  
\begin{align*}&\max(\|h_4^{(4)} \|, \|h_4^{(3)} \|\cdot \|h_4^{(1)} \|, \|h_4^{(2)} \|^2, \|h_4^{(2)} \| \cdot \|h_4^{(1)} \|^2, \|h_4^{(1)} \|^4 )\\
&=\max(24,(24\sqrt{2}-24)(4-2\sqrt{2}),(36-24\sqrt{2})^2,(36-24\sqrt{2})(4-2\sqrt{2})^2,(4-2\sqrt{2})^4)\\
&=24.
\end{align*}
 Letting  $M_m=2^{m-2}(m-1)!$ and noting that $M_1=\frac{1}{2}$, $M_2=1$, $M_3=4$ and $M_4=24$ means that we have now proved that inequalities (\ref{multihbd2yy}) and (\ref{multihbd2}) hold for $1\leq m\leq4$. 
%
\end{proof}
 
 \section{Proof of the general bounds of Section \ref{sec2}}\label{sec5}

\noindent{\emph{Proof of Proposition \ref{prop2.1}.}} 
Let $m \geq 1$, $z \in \mathbb{R}$ and $0<\alpha \leq 1.$ Let $M_m=2^{m-2}(m-1)!$ and $N_m=2^mM_m$. Also, let $h_{m,z,\alpha}\in\mathcal{F}_{m,z,\alpha}$ be such that $\max_{0\leq i\leq m}\|h_{m,z,\alpha}^{(i)}\|\leq\frac{N_m}{\alpha^m}$ (such a function exists by Corollary \ref{cormax}, if $0<\alpha\leq1$). Let $X$ be an arbitrary real-valued random variable and let $Y$ be a real-valued random variable with density $p(y)$, $y \in \mathbb{R}.$ Then
\begin{align}\label{eq:4}
 \mathbb{P}(X \leq z) - \mathbb{P}(Y \leq z)  & \leq   \mathbb{E}[h_{m,z,\alpha}(X)] - \mathbb{E}[h_{m,z,\alpha}(Y)] +  \mathbb{E}[h_{m,z,\alpha}(Y)] - \mathbb{P}(Y \leq z) \nonumber
 \\ & = \frac{N_m}{\alpha^m}\bigg\{\mathbb{E}\bigg[\frac{h_{m,z,\alpha}(X)}{N_m/\alpha^m}\bigg] - \mathbb{E}\bigg[\frac{h_{m,z,\alpha}(Y)}{N_m/\alpha^m}\bigg] \bigg\}  + \int^{z+\alpha}_z h_{m,z,\alpha} (y) p(y)\, dy \nonumber
 \\ & \leq \frac{N_m}{\alpha^m}d_{m}(X,Y)+ \int^{z+\alpha}_z h_{m,z,\alpha} (y) p(y) \,dy.
\end{align}

\vspace{2mm}

\noindent{(i)}:  
Suppose that $p(y) \leq A$ for all $y \in \mathbb{R}$. Then, from (\ref{eq:4}), we have that 
\begin{equation}\label{modeq}
\mathbb{P}(X \leq z)- \mathbb{P}(Y \leq z) \leq \frac{N_md_{m}(X,Y)}{\alpha^m} + \frac{A\alpha}{2},
\end{equation}
where we used that $\int^{z+\alpha}_z h_{m,z,\alpha}(y)\, dy=\frac{1}{2}$ (since the function we constructed to prove Corollary \ref{cormax} satisfies $h_{m,z,\alpha}(z+\frac{\alpha}{2}+x)+h_{m,z,\alpha}(z+\frac{\alpha}{2}-x)=1$). Taking $\alpha = (\frac{2N_m d_{m}(X,Y)}{A})^{\frac{1}{m+1}} $ (so that we require $d_m(X,Y)\leq\frac{A}{2N_m}$ to ensure $\alpha\leq1$) yields the bound $\mathbb{P}(X \leq z) - \mathbb{P}(Y \leq z) \leq 2(A^mM_m d_m(X,Y))^{\frac{1}{m+1}}. $ We can similarly obtain a lower bound which is the negative of the upper bound (this is also the case in the proofs of inequalities (\ref{dmsecond})--(\ref{dmfourth}) below), and so we have proved inequality (\ref{dmfirst}).

\vspace{2mm}

\noindent{(ii)}: Suppose now that the density of $Y$ has $n$ singularities $y_1,\ldots,y_n.$ 
Suppose also that $p(y) \leq -A\log |c(y-y_i)|$ for all $|y-y_i|<\epsilon\leq\frac{1}{c},$ $i=1,\ldots,n$, and that $\int_I p(y)\, dy\leq \int^{\delta}_{-\delta} -A \log |cy| \,dy$ for any interval $I$ of length $2\delta\leq2\epsilon.$ Assume that $0<\alpha \leq \min(1,2\epsilon).$ Then, on using these assumptions, inequality (\ref{eq:4}) and the basic inequality $h_{m,z,\alpha}(x)\leq1$, we obtain that 	
\begin{align*}
\mathbb{P}(X \leq z)-\mathbb{P}(Y \leq z) 
& \leq \frac{N_md_{m}(X,Y)}{\alpha^m} + \int^{\frac{\alpha}{2}}_{-\frac{\alpha}{2}} -A \log|cy|\, dy
\\ & = \frac{N_md_{m}(X,Y)}{\alpha^m} +A\alpha\bigg[1 +\log\bigg( \frac{2}{c\alpha}\bigg)\bigg].
\end{align*}
Taking $ \alpha = (\frac{N_m d_{m}(X,Y) }{A})^{\frac{1}{m+1}} $ now yields inequality (\ref{dmsecond}). 
 
\vspace{2mm}

\noindent{(iii)}: Suppose now $p(y) \leq A|y-y_i|^{-a}$ for all $|y-y_i|<\epsilon$, $i=1,\ldots,n$. Then arguing similarly to we did in part (ii) of the proof we get that
 \begin{equation*}
\mathbb{P}(X \leq z) - \mathbb{P}(Y \leq z) 
 \leq \frac{N_md_{m}(X,Y)}{\alpha^m} + \int^{\frac{\alpha}{2}}_{-\frac{\alpha}{2}} A|y|^{-a} \,dy
 = \frac{N_md_{m}(X,Y)}{\alpha^m} +\frac{2^a A}{1-a} \alpha^{1-a}.
\end{equation*}
Taking $\alpha = (\frac{N_m(1-a)d_m(X,Y)}{2^aA})^{\frac{1}{m+1-a}} $ yields inequality (\ref{dmthird}).

\vspace{2mm}

\noindent{(iv)}: Now suppose that $p(y) \leq A |y-y_i|^{-a}(-\log |c(y-y_i)|)^b $ for all $|y-y_i| < \epsilon$, $i=1,\ldots,n$. Then
 \begin{align}\label{eq:15}
\mathbb{P}(X \leq z)-\mathbb{P}(Y \leq z) & \leq \frac{N_md_{m}(X,Y)}{\alpha^m} + \int^{\frac{\alpha}{2}}_{-\frac{\alpha}{2}} A |y|^{-a}(-\log |cy|)^b \,dy \nonumber
\\ & = \frac{N_md_{m}(X,Y)}{\alpha^m} + 2\int^{\frac{\alpha}{2}}_{0} A y^{-a}(-\log (cy))^b \,dy \nonumber
\\ & = \frac{N_md_{m}(X,Y)}{\alpha^m} + \frac{2c^{a-1}A}{(1-a)^{b+1}} \Gamma\bigg(b+1,(1-a) \log \bigg(\frac{2}{c\alpha}\bigg) \bigg),
\end{align}
where $\Gamma(r,x) = \int^\infty_x t^{r-1} e^{-t} \,dt $ is the upper incomplete gamma function, and we made the change of variables $t=-(1-a) \log (cy)$ to compute the second integral. In order to obtain a simplified bound, we will apply the inequality $\Gamma(b+1,y) \leq 2^{b+1}y^b e^{-y}, $ which holds for $b \geq 0$ and $e^y>2^{b+1}$ (see 
\cite{jameson16}). 
Applying this inequality into (\ref{eq:15}) gives the bound
\begin{equation*} 
\mathbb{P}(X \leq z)-\mathbb{P}(Y \leq z) \leq \frac{N_md_{m}(X,Y)}{\alpha^m} + \frac{2^{a+b+1}A}{1-a} \alpha^{1-a}\log^b\bigg(\frac{2}{c\alpha}\bigg),
\end{equation*}
under the additional assumption $\alpha^{1-a}<\frac{1}{2^{a+b}c^{1-a}}.$
Now take $\alpha = (\frac{(1-a)N_m d_{m} (X,Y)}{2^{a+b+1}A})^{\frac{1}{m+1-a}} $. \hfill $\Box$





\vspace{3mm}

\noindent{\emph{Proof of Proposition \ref{prop2.2}.}} 
The proof is similar to that of Proposition \ref{prop2.1}, and we only highlight the ways in which the argument changes. Firstly, we modify the argument by using the bound $\max_{0 \leq i \leq m} \|h_{m,z,\alpha}^{(i)} \| \leq M_m(1+ \frac{2^m}{\alpha^m})$, which is valid for all $\alpha>0$ (see Corollary \ref{cormax}). Applying this bound leads to the following analogue of inequality (\ref{eq:4}) that holds for all $\alpha>0$:
\begin{equation*}\mathbb{P}(X \leq z) - \mathbb{P}(Y \leq z)\leq \bigg(M_m+\frac{N_m}{\alpha^m}\bigg)d_{m}(X,Y)+ \int^{z+\alpha}_z h_{m,z,\alpha} (y) p(y) \,dy.
\end{equation*}

\vspace{2mm}

\noindent{(i):} Arguing as we did in obtaining inequality (\ref{modeq}) we get
\begin{equation*}\mathbb{P}(X \leq z) - \mathbb{P}(Y \leq z) \leq \bigg(M_m+ \frac{N_m}{\alpha^m}\bigg)d_m(X,Y)+\frac{A\alpha}{2},
\end{equation*}
and taking $\alpha = (\frac{2N_md_m(X,Y)}{A})^{\frac{1}{m+1}}$ yields the desired bound.

\vspace{2mm}

\noindent{(ii):} We have that 
\begin{align}\label{inty}
\mathbb{P}(X \leq z)-\mathbb{P}(Y \leq z)  
& \leq \bigg(M_m+ \frac{N_m}{\alpha^m}\bigg)d_m(X,Y) + 2\int_{0}^{\frac{\alpha}{2}} -A \log_{-}(cy)\, dy+\frac{B\alpha}{2}.
\end{align}
Now, 
\begin{align*}\int_{0}^{\frac{\alpha}{2}} - \log_{-}(cy)\, dy=\int_{0}^{\frac{\alpha}{2}\wedge \frac{1}{c}} - \log(cy)\, dy&=\bigg(\frac{\alpha}{2}\wedge\frac{1}{c}\bigg)\bigg[1+\log\bigg(\frac{1}{c(\frac{\alpha}{2}\wedge\frac{1}{c})}\bigg)\bigg]\\
&\leq\frac{\alpha}{2}\bigg[1 +\log_{-}\bigg( \frac{2}{c\alpha}\bigg)\bigg].
\end{align*}
Applying this inequality to (\ref{inty}) and taking $ \alpha = (\frac{N_m d_{m}(X,Y) }{A})^{\frac{1}{m+1}} $ yields the desired bound.

\vspace{2mm}

\noindent{(iii):} We have that
\begin{equation*}
\mathbb{P}(X \leq z) - \mathbb{P}(Y \leq z) \leq \bigg(M_m+ \frac{N_m}{\alpha^m}\bigg)d_m(X,Y) +\frac{2^a A}{1-a} \alpha^{1-a},
\end{equation*}
and taking $\alpha = (\frac{N_m(1-a)d_m(X,Y)}{2^aA})^{\frac{1}{m+1-a}} $ yields the desired bound.

\vspace{2mm}

\noindent{(iv):} Arguing similarly to we did in part (ii) we obtain that, for $\alpha^{1-a}<\frac{1}{2^{a+b}c^{1-a}}$, 
\begin{equation*} 
\mathbb{P}(X \leq z)-\mathbb{P}(Y \leq z) \leq \bigg(M_m+ \frac{N_m}{\alpha^m}\bigg)d_m(X,Y)+ \frac{2^{a+b+1}A}{1-a} (c\alpha)^{1-a}\log_{-}^b\bigg(\frac{2}{c\alpha}\bigg)+\frac{B\alpha}{2},
\end{equation*}
and taking $\alpha = (\frac{(1-a)N_m d_{m} (X,Y)}{2^{a+b+1}A})^{\frac{1}{m+1-a}} $ yields the desired bound.
\hfill $\Box$

\vspace{3mm}

\noindent{\emph{Proof of Proposition \ref{prop2.4}.}} Let $m \geq 1$, $z \in \mathbb{R}$ and $\alpha>0.$ Let $M_m'$ be given by $M_m'=2^{m-2}(m-1)!$ for $1\leq m\leq 4$, and $M_m'=2^{\frac{3m}{2}-2}(m-1)!$ for $m\geq5$, and denote $N_m'=2^m M_m'$. 
Also, let $h_{m,z,\alpha}\in\mathcal{F}_{m,z,\alpha,d}$ be such that $\max_{0\leq i\leq m}|h_{m,z,\alpha}|_i\leq M_m'(1+\frac{2^m}{\alpha^m})$ (such a function exists by Lemma \ref{prop2.13}).  Let $X$ be an arbitrary $\mathbb{R}^d$-valued random vector and let $Y$ be a $\mathbb{R}^d$-valued random vector with density $p(y)$, $y \in \mathbb{R}.$ Then it is readily seen that a multivariate analogue of inequality (\ref{eq:4}) is given by 
\begin{equation}\label{near77}\mathbb{P}(X \leq z) - \mathbb{P}(Y \leq z)\leq \bigg(M_m'+\frac{N_{m}'}{\alpha^m}\bigg)d_m(X,Y)+\int_{\mathbb{R}^d} (h_{z,\alpha ,m}(y) - \mathbf{1}_{y \leq z}) p(y) \,dy.
\end{equation}
Now, let $r=(r_1,\ldots,r_d)^T \in \mathbb{R}^d $.  We can write 
\begin{equation*}
\int_{\mathbb{R}^d} (h_{z,\alpha ,m}(y) - \mathbf{1}_{y \leq z}) p(y) \,dy  = \int_{\mathcal{S}} (h_{z,\alpha ,m}(y) - \mathbf{1}_{y \leq z}) p(y) \,dy,
\end{equation*}    
where $\mathcal{S} = \mathbb{R}^d - \{ r \in \mathbb{R}^d: r_1 \leq z_1,\ldots,r_d \leq z_d  \} - \{ r \in \mathbb{R}^d: r_1 \geq z_1+\alpha \lor \ldots \lor r_d \geq z_d+\alpha  \}$. We now split the set $\mathcal{S}$ into several disjoint subsets and we will then calculate the integral over these subsets. Let $2^{\left[d \right]}$ be the set of all subsets of $[d] = \left\{1,2,\ldots,d \right\} .$ 

 Let $J \in 2^{[d]} \backslash \emptyset $ and $\mathcal{S}_J = \{ r \in \mathbb{R}^d: z_k \leq r_k < z_k + \alpha, \forall k \in J~ \text{and}~ r_k \leq z_k, \forall k \in J^c  \} $. For example, $\mathcal{S}_{\{1 \}} =\{ r \in \mathbb{R}^d : z_1 \leq r_1<z_1+\alpha ~\text{and}~ r_k \leq z_k, \forall k\in\{2,\ldots,d\} \}.  $ Using this notation, we are able to write $\mathcal{S}  =  \cup_{J \in (2^{[d]} \backslash \emptyset)} \mathcal{S}_J$, so that
\begin{equation*}
 \int_{\mathcal{S}} (h_{z,\alpha ,m}(y) - \mathbf{1}_{y \leq z}) p(y)\, dy  = \sum_{J \in (2^{[d]} \backslash \emptyset)} \int_{\mathcal{S}_J} h_{z,\alpha ,m}(y)p(y) \,dy.
\end{equation*}
Assume that $|J| = d' \leq d$  and $J' = \left\{1,2,\ldots,d' \right\}. $ Then we get that 
\begin{equation}\label{eq:6}
  \int_{S_{J'}}  h_{z,\alpha ,m}(y)p(y)\, dy =  \int_{z_1}^{z_1+\alpha} \cdots \int^{z_{d'}+\alpha}_{z_{d'}} (\prod^{d'}_{i=1} h_{m,z_i,\alpha}(y_i)) p'(y_1,\ldots,y_{d'}) \,dy_{d'} \cdots dy_1,
\end{equation}
where $ p'(y_1,\ldots,y_{d'}) = \int^{z_{d'+1}}_{-\infty} \cdots \int^{z_d}_{-\infty} p(y_1,\ldots,y_d) \,d_{y_d}\cdots d_{y_{d'+1}}.  $

\vspace{2mm}

\noindent{(i):} Suppose $p(y) \leq A$, for all $y \in \mathbb{R}^d$.  Then, using (\ref{eq:6}), we have $\int_{S_J}  h_{z,\alpha ,m}(y)p(y) \,dy \leq C_J \alpha^{d'}$ and hence $\int_{\mathcal{S}} (h_{z,\alpha ,m}(y) - \mathbf{1}_{y \leq z}) p(y) \,dy \leq \sum^d_{i=1} C_i \alpha^i\leq C(\alpha+\alpha^d)$, where $C>0$ is some constant, which now may change from one expression to the next.  Substituting this inequality into (\ref{near77}) and then taking $\alpha=(d_m(X,Y))^{\frac{1}{m+1}}$ implies that there exists a constant $C>0$ such that
\begin{equation*}\mathbb{P}(X \leq z) - \mathbb{P}(Y \leq z)\leq C\big[(d_m(X,Y))^{\frac{1}{m+1}}+d_{m}(X,Y)+(d_m(X,Y))^{\frac{d}{m+1}}\big].
\end{equation*}
Similarly, we can find a lower bound and so we deduce that
there exists a constant $C>0$ such that
\begin{equation*}d_{\mathrm{K}}(X,Y)\leq C\big[(d_m(X,Y))^{\frac{1}{m+1}}+d_{m}(X,Y)+(d_m(X,Y))^{\frac{d}{m+1}}\big].
\end{equation*}
To get a compact final bound, we note that if $d_{m}(X,Y)\leq1$ we get that (for another constant $C$) $d_{\mathrm{K}}(X,Y)\leq C(d_m(X,Y))^{\frac{1}{m+1}}$, whilst if $d_m(X,Y)>1$, then taking $C$ large enough so that $C(d_m(X,Y))^{\frac{1}{m+1}}\geq1$ means that $d_{\mathrm{K}}(X,Y)\leq C(d_m(X,Y))^{\frac{1}{m+1}}$ is satisfied. Therefore, there exists a universal constant $C>0$ such that $d_{\mathrm{K}}(X,Y)\leq C(d_m(X,Y))^{\frac{1}{m+1}}$.
  
\vspace{2mm}

\noindent{(ii) and (iii)}: Now we obtain that 
\begin{align*}\int_{\mathcal{S}} (h_{z,\alpha ,m}(y) - \mathbf{1}_{y \leq z}) p(y) \,dy \leq \sum^d_{i=1} \alpha^{i}\bigg[C_i \log_{-}\bigg(\frac{2}{c\alpha}\bigg) + D_i\bigg]\leq C_1'(\alpha+\alpha^d)\log\bigg(\frac{C_2'}{\alpha}\bigg)
\end{align*}
and
\[\int_{\mathcal{S}} (h_{z,\alpha ,m}(y) - \mathbf{1}_{y \leq z}) p(y) \,dy \leq\sum_{i=1}^d C_i\alpha^{i-a}\leq C(\alpha^{1-a}+\alpha^{d-a}),\]
where $C_1',C>0$ and $C_2'>1$ are universal constants. We now proceed as we did in proving part (i) to get the desired compact final bounds.
\hfill $\Box$

\vspace{3mm}

\noindent{\emph{Proof of Proposition \ref{prop2.5}.}} 
We prove the result for the case $\mu=0$; the general case follows by making a simple translation. Using inequality (\ref{eq:7}), we have that, for $\alpha>0$, 
\begin{align}
  \mathbb{P}(X \leq z )- \mathbb{P}(Y \leq z) & \leq \mathbb{E}[h_{z,\alpha ,m}(X)]-\mathbb{E}[h_{z,\alpha ,m}(Y)]+\mathbb{E}[h_{z,\alpha ,m}(Y)]-\mathbb{P}(Y \leq z) \nonumber
\\ & \leq \bigg(M_m'+\frac{N_{m}'}{\alpha^m}\bigg)d_{m}(X,Y) + \mathbb{P}(Y \leq z + \alpha  )- \mathbb{P}(Y \leq z) \nonumber
\\ 
\label{nearbb}& \leq \bigg(M_m'+\frac{N_{m}'}{\alpha^m}\bigg)d_{m}(X,Y) + \frac{\alpha}{\sigma}(\sqrt{2 \log d} +2).
\end{align}
Taking $\alpha = (\frac{\sigma N_{m}'}{\sqrt{2 \log d}+2} d_m (X,Y))^{\frac{1}{m+1}}$ yields inequality (\ref{multin0}).  If we now suppose that $0<\alpha\leq1$, then we can replace the factor $(M_m'+\frac{N_{m}'}{\alpha^m})$ in (\ref{nearbb}) by $\frac{N_{m}'}{\alpha^m}$. Taking the same $\alpha$ as we did to obtain inequality (\ref{multin0}) now yields (\ref{multin}). Finally, we note that if $m\geq2$ then inequalities (\ref{multihbd}) and (\ref{multihbd2}) hold for $0<\alpha\leq2$, allowing a weaker assumption for $d_m(X,Y)$ to be made. 
\hfill $\Box$

\appendix



\section{Explicit formulas for the functions in Lemma \ref{conj1} when $1\leq m\leq4$}\label{appcon}

For $1\leq m\leq4$, we have $h_m(x)=1$ for $x\leq-1$, and $h_m(x)=0$ if $x\geq1$. For $-1< x<1$,
\begin{align*}h_1(x)&=-\tfrac{1}{2}x+\tfrac{1}{2}, \quad x\in(-1,1),\\
h_2(x)&=\left\{ \begin{aligned} 1-\tfrac{1}{2}(1+x)^2, & & \text{$x \in (-1,0]$,}
 \\ \tfrac{1}{2}(1-x)^2, & & \text{$x \in (0,1)$,}
 \end{aligned} \right.\\
h_3(x)&= \left\{ \begin{aligned}  1-\tfrac{2}{3}(1+x)^3, & & \text{$x \in (-1,-\tfrac{1}{2}]$,} &
 \\ \tfrac{2}{3}x^3-x+\tfrac{1}{2}, & & \text{$x \in (-\tfrac{1}{2},\tfrac{1}{2}]$,}&
 \\ \tfrac{2}{3}(1-x)^3, & & \text{$x \in (\tfrac{1}{2},1)$,}&
 \end{aligned} \right.\\
 h_4(x)&=\left\{ \begin{aligned} 
    1- (1+x)^4, & & \text{$x \in (-1,-\tfrac{1}{\sqrt{2}}]$,} &
 \\ x^4+4(\sqrt{2}-1)x^3+2(\sqrt{2}-2)x+\tfrac{1}{2}, & & \text{$x \in (-\tfrac{1}{\sqrt{2}},0]$,} &
 \\ -x^4+4(\sqrt{2}-1)x^3+2(\sqrt{2}-2)x+\tfrac{1}{2}, & & \text{$x \in (0,\tfrac{1}{\sqrt{2}}]$,} &
 \\  (1-x)^4, & & \text{$x \in (\tfrac{1}{\sqrt{2}},1)$.} &
 \end{aligned} \right. 
\end{align*}

\section*{Acknowledgements}
\rg{We would like to thank the reviewer for their constructive comments and helpful suggestions.} RG was supported by a Dame Kathleen Ollerenshaw Research
Fellowship. SL was supported by a University of Manchester Research Scholar Award.

\appendix

\normalsize

\end{document}